\documentclass[11pt]{article}

\usepackage[margin=1.25in]{geometry}
\usepackage{multicol}
\usepackage{hyperref}
\usepackage{amsmath}
\usepackage{amscd}
\usepackage{amssymb}
\usepackage{amsthm}
\usepackage{array}
\usepackage{graphicx}
\usepackage{xy}     
\usepackage{comment}
\usepackage{fancyhdr}
\usepackage{tikz}
\usepackage{tikz-cd}

\theoremstyle{plain}
\newtheorem{theorem}{Theorem}[section]
\newtheorem{lemma}[theorem]{Lemma}
\newtheorem{proposition}[theorem]{Proposition}
\newtheorem{corollary}[theorem]{Corollary}
\newtheorem{claim}{Claim}

\theoremstyle{definition}
\newtheorem{remark}[theorem]{Remark}
\newtheorem{example}{Example}
\newcommand{\Out}{\mathrm{Out}}
\renewcommand{\S}{\mathcal{S}}
\newcommand{\FF}{\mathcal{FF}}
\newcommand{\F}{\mathbb{F}}
\newcommand{\K}{\mathcal{K}}
\newcommand{\X}{\mathcal{X}}
\newcommand{\Mod}{\mathrm{Mod}}

\newcommand{\fact}{\textsc{fact}}

\newcommand{\syl}{\mathrm{syl}}

\begin{document}

\title{Right-angled Artin groups and Out$(\mathbb{F}_n)$ I: \\ quasi-isometric embeddings}
\author{Samuel J. Taylor  
\\
\small \texttt{staylor@math.utexas.edu}\\
} 
\maketitle


\begin{abstract}
We construct quasi-isometric embeddings from right-angled Artin groups into the outer automorphism group of a free group. These homomorphisms are in analogy with those constructed in \cite{CLM}, where the target group is the mapping class group of a surface. Toward this goal, we develop tools in the free group setting that mirror those for surface groups as well as discuss various analogs of subsurface projection; these may be of independent interest. 
\end{abstract}

\section{Introduction}
For a finite simplicial graph $\Gamma$ with vertex set $\Gamma^0$, the right-angled Artin group $A(\Gamma)$ is the group presented with generators $s_i \in \Gamma^0$ and relators $[s_i,s_j] =1$ when  $s_i$ and $s_j$ are joined by an edge in $\Gamma$. Such groups, though simple to define, have been at the center of recent developments in geometric group theory and low-dimensional topology because of both the richness of their subgroups (and their residual properties) and the frequency with which they  occur as subgroups of other geometrically significant groups. For example, Wang constructed injective homomorphisms from certain right-angled Artin groups into SL$_n(\mathbb{Z})$, for $n \ge 5$ \cite{Wang}. In \cite{KapRAAGs}, Kapovich proved that for any finite simplicial graph $\Gamma$ and any symplectic manifold $(M,\omega)$, $A(\Gamma)$ embeds into the group of Hamiltonian symplectomorphisms of $(M, \omega)$. Koberda showed that in the mapping class group of a surface raising any collections of mapping classes that are supported on connected subsurfaces to a suitably high power generates a right-angled Artin group \cite{Kobraag}.  In \cite{CLM}, the authors constructed quasi-isometric embeddings of right-angled Artin groups into mapping class groups using partial pseudo-Anosov mapping classes supported on either disjoint or overlapping subsurfaces, depending on whether the mapping classes represented commuting or non-commuting vertex generators. Specifically, they prove the following:

\begin{theorem}[Theorem 1.1 of \cite{CLM}]
Suppose that $f_1, \ldots f_n \in \Mod(S)$ are fully supported on disjoint or overlapping non-annular subsurfaces. Then after raising to sufficiently high powers, the elements generate a quasi-isometrically embedded right-angled Artin subgroup of $\Mod(S)$. Furthermore, the orbit map to  Teichm\"{u}ller space is a quasi-isometric embedding. 
\end{theorem}

\begin{corollary}[Corollary 1.2 of \cite{CLM}]
Any right-angled Artin group admits a homomorphism to some mapping class group which is a quasi-isometric embedding, and for which the orbit map to Teichm\"{u}ller space is a quasi-isometric embedding.
\end{corollary}

In this paper, we develop the theory necessary to quasi-isometrically embed right-angled Artin groups into Out$(\mathbb{F}_n)$, in analogy with \cite{CLM} for the mapping class group. Here, we show the following (see Section \ref{homomorphisms} for definitions and a more general statement):

\begin{theorem}
Suppose that $f_1, \ldots f_n \in \Out(\F_n)$ are fully supported on an admissible collection of free factors. Then after raising to sufficiently high powers, the elements generate a quasi-isometrically embedded right-angled Artin subgroup of $\Out(\F_n)$. Furthermore, the orbit map to Outer space is a quasi-isometric embedding.
\end{theorem}

Recall that \emph{Outer space} $\X_n$, introduced in \cite{CVouter}, is the space of $\F_n$-marked metric graphs, and the distance being considered is the asymmetric Lipschitz metric of \cite{FMout} (see Section \ref{appen}). The \emph{admissible collection} condition on the set of free factors in our main theorem is meant to mimic the surface case where the subsurfaces considered are either disjoint or overlap. It should be noted that similar to the mapping class group case our conditions for generating a right-angled Artin group involve translation lengths of the outer automorphims on the free factor complexes of their supporting free factors. Further, if $\Gamma$ is the ``coincidence graph'' for the involved free factors, then the right-angled Artin group generated is $A(\Gamma)$. This is made precise in Section \ref{homomorphisms}. We also obtain,

\begin{corollary}
Any right-angled Artin group admits a homomorphism to $\Out(\F_n)$, for some $n$, which is a quasi-isometric embedding, and for which the orbit map to Outer space is a quasi-isometric embedding.
\end{corollary}

We remark that although much of the inspiration for this paper is drawn from \cite{CLM}, there are several significant points of departure. First, as opposed to subsurface projections in the mapping class group situation, when working with $\Out(\mathbb{F}_n)$ there are different possible projections that one could employ. In \cite{BFproj}, the authors define the projection of the free factor $A$ to the free splitting complex of the free factor $B$ when the two factors are  in ``general position.'' These projections, though powerful in other settings, are not delicate enough for our application. In particular, the presence of commuting outer automorphisms in our construction precludes free factors from satisfying the conditions for finite diameter Bestvina-Feighn projections, as would be required.  In \cite{SS}, a different sort of projection is considered; the authors project a sphere system in $M_n$, the doubled handlebody of genus $n$, to the disk and sphere complex of certain submanifolds. Though interesting in their own right, these projections do not always give free splittings of free factors and so are not used in this paper. These difficulties are discussed in detail in Section \ref{relations}. To resolve these issues we develop our own projections which are tailored for the applications in this paper. In the process, we demonstrate the relationship between the projections of \cite{BFproj} and \cite{SS}, answering a question appearing in both papers.

Second, unlike in the mapping class group case, there are no immediate analogs of the Masur-Minsky formulas, giving coarse estimates of distance in $\Mod(S)$, that apply to our construction. Instead, in Section \ref{lower} we use the partial ordering on the syllables of a word in the standard generators of $A(\Gamma)$ to control the order in which distance can be made when projecting a geodesic in $\Out(\mathbb{F}_n)$ to the free factor complex of a free factor. This suffices for providing the lower bounds on $\Out(\mathbb{F}_n)$-distance in our main theorem.

Finally, we note that there is another method to construct quasi-isometrically embedded right-angled Artin subgroups of $\Out(\F_{2g})$. One could start with a once punctured genus $g$ surface $\dot{S}$ and using the methods of \cite{CLM} build a quasi-isometric embedding from $A(\Gamma)$ into $\Mod(\dot{S})$. In \cite{HamHen}, the authors show that the injective homomorphism $\Mod(\dot{S}) \to \Out(\F_{2g})$, induced by the action of $\Mod(\dot{S})$ on $\pi_1(\dot{S})= \F_{2g}$, is itself a quasi-isometric embedding. Composing two such maps then gives a quasi-isometric embedding from $A(\Gamma)$ into $\Out(\F_{2g})$. These homomorphisms, however, have the property that they factor through mapping class groups and, hence, fix the conjugacy class in $\F_{2g}$ corresponding to the puncture. In our approach, homomorphisms into $\Out(\F_n)$ do not factor through mapping class groups and outer automorphism are supported on free factors, rather than cyclic factors. Further, our construction produces homomorphisms into $\Out(\F_n)$ that have quasi-isometric orbit maps into Outer space.

\subsection{Plan of paper}
The paper is organized as follows: Section $2$ covers background that is needed throughout the paper. Section $3$ defines the subfactor projections that we use, gives their basic properties, and relates them to the projections of \cite{BFproj}. Section $4$ defines the homomorphisms of right-angled Artin groups into $\Out(\F_n)$ that are of interest, gives a precise statement of our main theorem, and provides a few examples. 

Central to the proof of our main theorem is a version of Behrstock's inequality which controls the projection of a tree into the free factor complex of two free factors that \emph{overlap}. To prove this, we work with a topological model of the projections developed in Section $5$. This approach has the additional advantage that the methods developed in this section can be used to relate various notions of projection that exist in the literature. In Section $6$, we prove the version of Behrstock's inequality that is needed. This is followed by Sections $7$ and $8$ which give related notions of order for both free factors and syllables of $g \in A(\Gamma)$, respectively. This is in preparation for Section $9$ which closely follows the arguments of \cite{CLM} and gives conditions when normal form words in $A(\Gamma)$ give large projections to the free factor complex of (certain) free factors.

Having arranged large projection distances, the last step is to argue that for ``non-disjoint'' free factors these distances independently contribute to distance in $\Out(\F_n)$; this is done in Section $10$. This section can be thought of as making up for the lack of lower bounds coming from Masur-Minsky type formulas. The proof that our homomorphisms are quasi-isometric embeddings into $\Out(\F_n)$ is then concluded in Section $11$. Since showing that these homomorphisms give quasi-isometric orbit maps into Outer space involves different terminology and constants, this is done in the appendix.

\subsection{Acknowledgments} 

The author is grateful to Patrick Reynolds, Chris White, and Nick Zufelt for helpful conversation. Thanks are also due to Chris Leininger and Chris Westenberger, who made constructive comments on an earlier version of this paper. Most importantly, the author is indebted to his advisor Alan Reid as well as Hossein Namazi for each's encouragement, advice, and continuous feedback throughout this project.

\section{Background}

\subsection{Quasi-isometries}
Let $(X,d_X)$ and $(Y,d_Y)$ be metric spaces. A map $f: X \to Y$ is a \emph{$(K,L)$ - quasi-isometric embedding} if for all $x_1,x_2 \in X$
$$\frac{1}{K}d_X(x_1,x_2) -L \le d_Y(f(x_1),f(x_2)) \le K d_X(x_1,x_2) +L. $$
In addition, if every point in $Y$ is within distance $L$ from the image $f(X)$, then $f$ is a \emph{quasi-isometry} and $X$ and $Y$ are said to be \emph{quasi-isometric}. In this paper, the metric spaces of interest arise from finite dimensional simplicial complexes where a fixed complex is considered with the metric induced by giving each simplex the structure of a Euclidean simplex with side length one. Recall that if $K$ is a finite dimensional simplicial complex, then this piecewise Euclidean metric on $K$ is quasi-isometric to $K^1$, the $1$-skeleton of $K$, with its standard graph metric \cite{BH}. Since we are interested in the coarse geometry of such complexes, i.e. their metric structure up to quasi-isometry, this justifies our convention of when working with a complex $K$ to consider only the graph metric on $K^1$, rather than the entire complex. Here, and below, a \emph{graph} is a $1$-dimensional CW complex and a simply connected graph is a \emph{tree}. 

\subsection{$\Out(\mathbb{F})$ basics}\label{basics}
Fix $n \ge 2$ and let $\mathbb{F}_n$ denote the free group of rank $n$ and $\Out(\mathbb{F}_n)$ its group of outer automorphisms. When clear from context, the subscript $n$ will be dropped from the notation. In this section we recall the definition and basic facts concerning some relevant $\Out(\mathbb{F)}$-complexes. First, a \emph{splitting} of $\mathbb{F}$ is a minimal, simplicial actions $\mathbb{F} \curvearrowright T$ on a non-trivial simplicial tree, the action being determined by a homomorphism $\psi: \F \to \mathrm{Aut}(T)$ into the simplicial automorphisms of $T$. An action on a tree is \emph{minimal} if there is no proper invariant subtree. By a \emph{free splitting}, we mean a splitting with trivial edge stabilizers and refer to a \emph{$k$-edge splittings} as a free splittings with $k$ natural edge orbits. From Bass-Serre theory, $k$-edge splittings correspond to graph of groups decompositions of $\mathbb{F}$ with $k$ edges, each edge with trivial edge group. Two actions $\mathbb{F} \curvearrowright T$ and $\mathbb{F} \curvearrowright T'$ are \emph{conjugate} if there is a $\mathbb{F}$-equivariant homeomorphism $\chi:T \to T'$ and the conjugacy class of an action is denoted by $[\mathbb{F} \curvearrowright T]$. We will usually drop the action symbol from the notation and refer to the splitting by $T$. Finally, an equivariant surjection $c: T \to T'$ between $\mathbb{F}$-trees is a \emph{collapse map} if all point preimages are connected. In this case, $T$ is said to be a refinement of $T'$.

The \emph{free splitting complex} $\mathcal{S}_n$ of the free group $\mathbb{F}_n$ is the simplicial complex defined as follows (see \cite{HMsplit} for details): The vertex set $\mathcal{S}_n^0$ is the set of conjugacy classes of $1$-edge splittings of $\mathbb{F}_n$, and $k+1$ vertices $[T_0], \ldots, [T_k]$ determine a $k$-simplex of $\mathcal{S}_n$ if there is a $(k+1)$-edge splitting $T$ and collapse maps $c_i: T \to T_i$, for each $i=0,\ldots,k$. That is, a collection of vertices span a simplex in $\S_n$ if they have a common refinement. We will mostly work with the barycentric subdivision of the free splitting complex, denoted by $\mathcal{S}'_n$, whose vertices are conjugacy classes of free splittings of $\mathbb{F}_n$ and two vertices are joined by an edge if, up to conjugacy, one refines the other. Higher dimensional simplicies are determined similarly.

For $n\ge 3$, the \emph{free factor complex} $\mathcal{FF}_n$ of $\mathbb{F}_n$ is the simplicial complex defined as follows (see \cite{HVff} or \cite{BFhyp} for details): The vertices are conjugacy classes of free factors of $\mathbb{F}_n$ and $k+1$ conjugacy classes $[A_0], \ldots, [A_n] $ span a $k$-simplex if there are representative free factors in these conjugacy classes with $A_0 \subset A_1 \subset \ldots \subset A_n$. When $n=2$, the definition is modified so that $\FF_2$ is the standard Farey graph. In this case, vertices of $\FF_2$ are conjugacy classes of rank $1$ free factors and two vertices are jointed by an edge if there are representatives in these conjugacy classes that form a basis for $\F_2$.

$\Out(\mathbb{F}_n)$ acts simplicially on these complexes. For $\mathcal{FF}$, if $f \in \Out(\mathbb{F})$ is represented by an automorphism $\phi$, we define $f[A] =[\phi A]$. It is clear that this is independent of choice of $\phi$ and extends to a simplical action on all of $\mathcal{FF}$. For $\mathcal{S}$ the action is defined as follows: with $f$ and $\phi$ as above and $[T] \in \mathcal{S}^0$ with the action on $T$ given by the homomorphism $\psi: \mathbb{F} \to \text{Aut}(T)$, then $f[T]$ is the conjugacy class of $\mathbb{F}$-tree determined by $\psi \circ \phi^{-1}: \mathbb{F} \to \text{Aut}(T).$ That is, the underlying tree is unchanged and the action is precomposed with the inverse of a representative automorphism for $f$. Again, checking that this is a well-defined action that extends to all of $\mathcal{S}$ (or $\mathcal{S}'$) is an easy exercise. These definitions have the convenient property that if $[T]$ is a conjugacy class of free splitting with vertex stabilizers $[A_1], \ldots, [A_l] $, then $f[T]$ has vertex stabilizers $f[A_1], \ldots, f[A_l] $, for any $f  \in \Out (\mathbb{F})$. 

There is a natural, coarsely defined map $\pi: \mathcal{S}' \to \mathcal{FF}$. For $T \in \S'$, we set $\pi(T)$ equal to the set of free factors that arise as a vertex group of a $1$-edge collapse of $T$. That is, $A \in \pi(T)$ if and only if there is a tree $T_0 \in  \S$, $T$ refines $T_0$, and $A$ is a vertex group of $T_0$. We are content to have this map defined only on the vertices of $\S'$ and to observe property that for any $T \in \S'$, diam$_{\FF}(\pi(T)) \le 4$. Letting $d_{\FF}$ denote distance in $\FF$ and setting $d_{\FF}(\pi(T),\pi(T')) = $ diam$_{\FF}(\pi(T) \cup \pi(T'))$, it is easily verified that $\pi$ is $4$-Lipschitz. Note that here, and throughout the paper, the brackets that denote conjugacy classes of trees and free factors will often be suppressed when it should cause no confusion to do so.

Recent efforts to understand the free splitting and free factor complex have focused on their metric properties along with their similarity to the complex of curves of a surface. In particular, both complexes are now known to be Gromov-hyperbolic. Hyperbolicity of the free factor complex was first shown by \cite{BFhyp} then \cite{KapRaf}; hyperbolicity of the free splitting complex was first shown by \cite{HMsplit} and then \cite{HHhyp}. We record this as a single theorem.

\begin{theorem}[\cite{BFhyp, HMsplit, KapRaf, HHhyp}] 
For $n \ge 2$, $\S_n$ and $\FF_n$ are Gromov- hyperbolic.
\end{theorem}

Remark that the action $\Out(\mathbb{F}_n) \curvearrowright \mathcal{S}_n$ is far from proper; all vertices have infinite point stabilizers. There is, however, an invariant subcomplex of $\S'_n$ that is locally finite, and the inherited action is proper. This is the spine of Outer space and we refer the reader to  \cite{CVouter} or \cite{HMdistort} for details beyond what is discussed here. Also, see  \cite{Hatcher} or \cite{AS} for an alternative perspective. 

The \emph{spine of Outer space} $\K_n$ is the subcomplex of $\S'_n$ span by vertices that correspond to proper splittings of $\F_n$. Recall that a splitting $T$ is \emph{proper} if no element of $\F_n$ acts elliptically, i.e. fixes a vertex, in $T$. Hence, $T \in \S'_n$ is proper if $T/ \F_n$ is a graph with fundamental group isomorphic to $\F_n$. Observe that since $\Out(\F_n)$ preserve the vertices of $\S'_n$ corresponding to proper splittings there is an induced simplicial action $\Out(\F_n) \curvearrowright \K_n$.

 It is well-known that $\mathcal{K}_n$ is a locally finite, connected complex and that the action $\Out(\mathbb{F}_n) \curvearrowright \mathcal{K}_n$ is proper and cocompact (see \cite{CVouter}). Hence, for any tree $T \in \mathcal{K}_n^0$, the orbit map $g \mapsto gT$ defines a quasi-isometry from $\Out(\F_n)$ to $\K_n$ by the \u{S}varc-Milnor lemma \cite{BH}. As remarked above, the metric considered is the standard graph metric on $\K_n^1$, the $1$-skeleton of the spine of Outer space. This metric on $\mathcal{K}_n^1$ will serve as our geometric model for $\Out(\mathbb{F}_n)$. 

\subsection{The sphere complex}
We recall the $\Out(\mathbb{F}_n)$-equivalent identification between the free splitting complex and the sphere complex. See \cite{AS} for details. Take $M_n$ to be the connected sum of $n$ copies of $S^1 \times S^2$, or equivalently, the double of the handlebody of genus $n$. Let $M_{n,s}$ be $M_n$ with $s$ open $3$-balls removed. Note that $\pi_1M_n$ is isomorphic to $\F_n$ and once and for all fix such an isomorphism. A sphere $S$ in $M_{n,s}$ is \emph{essential} if it is not boundary parallel and does not bound a $3$-ball, and a collection of disjoint, essential, pairwise non-isotopic spheres in $M_n$ is called a \emph{sphere system}.  By \cite{Laud}, spheres $S_1$ and $S_2$ are homotopic in $M_n$ if and only if they are isotopic.

The \emph{sphere complex} $\S(M_n)$ is the simplicial complex whose vertices are isotopy classes of essential spheres and vertices $[S_0], \ldots, [S_k] $ span a $k$-simplex if there are representatives in these isotopy classes that are disjoint in $M_n$. It is a theorem of \cite{Laud} that setting $\Mod(M_n) = \pi_0$(Diff$M_n$) there is an exact sequence 
$$1 \to K \to \Mod(M_n) \to \Out(\F_n) \to 1, $$
where $K$ is a finite group generated by ``Dehn twists'' about essential spheres. Since elements of $K$ act trivially on $\S(M_n)$, we have a well-defined action $\Out(\F_n) \curvearrowright \S(M_n)$. The following proposition of Arramayona and Souto identifies $\S_n$ and $\S(M_n)$. See Section \ref{surfacesandsplittings} for how one constructs splittings from essential spheres.

\begin{proposition}[\cite{AS}] \label{AS}
For $n \ge2$, $\S_n$ and $\S(M_n)$ are $\Out(\F_n)$-equivariantly isomorphic.
\end{proposition}

\subsection{Translation length in $\mathcal{FF}_n$}
\emph{Fully irreducible} outer automorphisms are those $f \in \Out(\F_n)$ that have no positive power that fixes a conjugacy class of free factor, i.e. for any $A \in \FF_n$, $f^n(A) = A$ implies that $n=0$. Recall that the \emph{(stable) translation length} of an outer automorphism $f \in \Out(F_n)$ on $\FF_n$ is defined as
$$\ell_{\FF}(f) = \lim_{n \to \infty} \frac{d_{\FF}(A,f^nA)}{n} $$
where $A \in \FF_n$. It is not difficult to verify that $\ell_{\FF}(f)$ is well-defined, independent of choice of $A \in \FF_n$, and satisfies the property $\ell_{\FF}(f^n)= n \cdot \ell_{\FF}(f)$ for $n \ge 0$. Further note that $\ell_{\FF}(f) \ge c$ if and only if for all $A \in \FF_n$, $d_{\FF}(A, f^nA)\ge c|n|$. The following proposition characterizes those outer automorphisms with positive translation length on $\FF_n$.

\begin{proposition}[\cite{BFhyp}] \label{fftrans}
Let $f \in \Out(\mathbb{F}_n)$, $f$ is fully irreducible if and only if $\ell_{\FF}(f) > 0$.
\end{proposition}

It appears to be an open question whether there is a uniform lower bound on translation length for fully irreducible outer automorphisms for a fixed rank free group, as is the case for pseudo-Anosov mapping classes acting on the curve complex \cite{MM1}. It is worthwhile to note that when $n=2$, Proposition \ref{fftrans} reduces to the statement that if an outer automorphism is infinite order and does not fix a conjugacy class of a primitive element in $\F_2$, then it acts hyperbolically on $\FF_2$, which as noted above is the Farey graph. This well-known statement is what is used in most of our applications.

\section{Projections to free factor complexes}
For a finitely generated subgroup $H\le \F$, let $\S(H)$ and $\mathcal{F}(H)$ denote the free splitting complex and free factor complex of $H$, respectively. The subgroup $H$ is \emph{self-normalizing} if $N(H)=H$, where $N(H)$ is the normalizer of $H$ in $\F$. Less formally, a subgroup $H$ is self-normalizing if the only elements that conjugate $H$ back to itself are those elements in $H$. When $H$ is self-normalizing the complexes $\S(H)$ and $\mathcal{F}(H)$ depend only on the conjugacy class of $H$ in $\F$. More precisely, if $H' =gHg^{-1}$ for $g\in \F$, then $g$ induces an isomorphism between $\S(H)$ and $\S(H')$ (and between $ \mathcal{F}(H)$ and $\mathcal{F}(H')$) via conjugation. For any other $x \in \F$ with $H' = xHx^{-1}$ we see that $x^{-1}g$ normalizes $H$ and so $x^{-1}g \in H$. In this case, $gH=xH$ and it is easily verified that $g$ and $x$ induce identical isomorphisms between $\S(H)$ and $\S(H')$. Hence, when $H$ is self-normalizing we obtain a canonical identification between the free splitting complex of $H$ and the free splitting complex of each of its conjugates. The same holds for the free factor complex of $H$. This allows us to unambiguously refer to the free splitting complex or free factor complex for the conjugacy class $[H]$.  Finally, recall that a subgroup $C \le \F$ is \emph{malnormal} if $xCx^{-1} \cap C \neq \{1\}$ implies that $x \in C$. For example, free factors of $\F$ are malnormal and malnormal subgroups are self-normalizing.

\subsection{Projecting trees}
Given a free splitting $T \in \S'$ and a finitely generated subgroup $H \le \F$ denote by $T^H$ the \emph{minimal $H$-subtree} of $T$. This is the unique minimal $H$-invariant subtree of the restricted action $H \curvearrowright T$ and is either trivial, in which case $H$ fixes a unique vertex in $T$, or is the union of axes of elements in $H$ that act hyperbolically on $T$. When $T^H$ is not trivial, we define the \emph{projection} of $T$ to the free splitting complex of $H$ as $\pi_{\mathcal{S}(H)}(T) = [H \curvearrowright T^H]$, where the brackets denote conjugacy of $H$-trees. Note that this projection is a well-defined vertex of $\mathcal{S}'(H)$ and depends only on the conjugacy class of $T$. This is the case since any conjugacy between $\F$-trees will induce a conjugacy between their minimal $H$-subtrees. Further define the projection to the free factor complex of $H$ to be the composition $\pi_H(T) = \pi (\pi_{\mathcal{S}(H)}(T))$, where $\pi: \S(H) \to \FF(H)$ is the $4$-Lipschitz map defined in Section \ref{basics}. Hence, $\pi_H(T) \subset \mathcal{F}(H)$ is the collection of free factors of $H$ that arise as a vertex group of a one-edge collapse of the splitting $H \curvearrowright T^H$. When $H$ is also self-normalizing, e.g. a free factor, these projections are independent of the choice of $H$ within its conjugacy class. The following lemma verifies that such projections are coarsely Lipschitz. 

\begin{lemma}\label{refine}
Let $\mathbb{F}_n \curvearrowright T$ be a free splitting and $H \le \F_n$ finitely generated with $T^H$ non-trivial. Let $T_0$ be a refinement of $T$ with equivalent collapse map $c: T_0 \to T$. Then there is an induced map $c_H: T_0^H \to T^H$ which is also a collapse map. Hence, $T_0^H$ is a refinement of $T^H$.
\end{lemma}

\begin{proof}
Since $c(T_0^H) \subset T$ is an invariant $H$-tree, it contains $T^H$.  Also, the axis in $T_0$ of any hyperbolic $h\in H$, which exists because $T_0^H$ is nontrivial, is mapped by $c$ to either $h$'s axis in $T$ or a singe vertex stabilized by $h$; each of which is contained in $T^H$. Since $T_0^H$ is the union of such axes, we see that $c(T_0^H) = T^H$. Hence the map $c_H$ described in the lemma is given by restriction. It remains to show that $c_H$ is a collapse map. This is the case since for any $p \in T^H$,
$$c_H^{-1}(p) = T_0^H \cap c^{-1}(p) $$
is the intersection of two subtrees of $T_0$ and is therefore connected. 
\end{proof}

For a free factor $A$ of $\F$ we use the symbol $d_A$ to denote distance in $\mathcal{F}(A)$ and for $\mathbb{F}_n$-trees $T_1,T_2$ we use the shorthand 
$$d_A(T_1,T_2) := d_A (\pi_A (T_1), \pi_A(T_2)) = \text{diam}_A(\pi_A (T_1) \cup \pi_A(T_2))$$
when both projections are defined. The following proposition follows immediately from the definitions in this section and Lemma \ref{refine}.

\begin{proposition}[Basic properties I] \label{properties1}
Let $T_1,T_2$ be adjacent vertices in $\mathcal{K}_n$, $A \in \FF_n$, and $H$ a finitely generated and self-normalizing subgroup of $\F_n$ containing $A$, up to conjugacy. Then we have the following:
\begin{enumerate}
\item diam$_{\mathcal{F}(A)}(\pi_A(T)) \le 4,$
\item $d_A (T_1,T_2) \le 4,$
\item $\pi_A(T_1) = \pi_A(\pi_{\S(H)}(T_1))$ and so $d_A (T_1,T_2) = d_A (\pi_{\S(H)}(T_1),\pi_{\S(H)}(T_2)).$
\end{enumerate}
\end{proposition}

\subsection{Projecting factors}\label{factorproj} 
Consider rank $\ge 2$ free factors $A$ and $B$ of the free group $\mathbb{F}_n$. Define $A$ and $B$ to be \emph{disjoint} if they are distinct vertex groups of a common splitting of $\mathbb{F}_n$. Disjoint free factors are those that will support commuting outer automorphisms in our construction. Define $A$ and $B$ to \emph{meet} if there exist representatives in their conjugacy classes whose intersection is nontrivial and proper in each factor. In this section, we show that this intersection provides a well-defined projection of $[B]$ to $\mathcal{F}(A)$, the free factor complex of $A$. Note that if $A$ and $B$ meet, then $d_{\mathcal{FF}}([A],[B]) = 2$.

Fix free factors $A$ and $B$ in $\F_n$. Define the projection of $B$ into $\mathcal{F}(A)$ as
$$\pi_A(B) = \{ [A \cap gBg^{-1}] : g\in \mathbb{F}_n \} \setminus \{1\}, $$
where conjugacy is taken in $A$. Observe that $A$ and $B$ meet exactly when $\pi_A(B) \neq \emptyset \neq \pi_B(A)$.  We show that members of $\pi_A (B)$ are vertex groups of a single (non-unique) free splitting of $A$ and so $\pi_A(B)$ has diameter $\le 4$ in $\mathcal{F}(A)$. Since the projection is independent of the conjugacy class of $B$, this provides the desired projection from $[B]$ to the free factor complex of $A$. 

\begin{lemma}\label{welldefined}
Suppose the free factors $A$ and $B$ meet. Then diam$_{\mathcal{F}(A)} \pi_A(B) \le 4$.
\end{lemma}

\begin{proof}
First, observe that $g$ uniquely determines the class $[A \cap gBg^{-1}] \in \pi_A(B)$ up to double coset in $\mathbb{F}$. Precisely, $[A \cap gBg^{-1}] = [A \cap hBh^{-1}] \neq 1$ if and only if $AgB = AhB$; this follows from the fact that free factors are malnormal. Now choose any marked graph $G$ which contains a subgraph $G^B$ whose fundamental group represents $B$ up to conjugacy. Let $p_A: \tilde{G}_A \to G$ be the cover of $G$ corresponding to free factor $A$ and let $G_A$ denote the core of $\tilde{G}_A$. By covering space theory, the components of $p^{-1}(G^B)$ are in bijective correspondence with the double cosets $\{AgB: g \in \mathbb{F} \}$ and the fundamental group of the component corresponding to $AgB$ is $A \cap gBg^{-1}$. Since the core carries the fundamental group of $\tilde{G}_A$, all nontrivial subgroups $A \cap gBg^{-1}$ correspond to double cosets representing components of $p^{-1}(G^B)$  in the core $G_A$. Hence, $G_A$ is a marked $A$-graph that contains disjoint subgraphs whose fundamental groups (up to conjugacy in $A$) are the subgroups of $\pi_A(B)$. This completes the proof.
\end{proof}

If $A \in \FF_n$ and $f \in \Out(\F_n)$ stabilizes $A$, then $f$ induces an outer automorphism of $A$, denoted $f|_A \in \Out(A)$. In this case, let $\ell_A(f)$ represent the translation length of $f|_A$ on $\mathcal{F}(A)$. By Proposition \ref{fftrans}, if $f|_A$ is fully irreducible in $\Out(A)$, then $\ell_{A}(f) >0$. The following proposition provides the addition properties of the projections that will be needed throughout the paper. Its proof is a straightforward exercise in working through the definition of this section.  

\begin{proposition}[Basic Properties II] \label{properties2}
Let $A,B,C \in \FF_n$ so that $A$ and $B$ meet and $A$ and $C$ are disjoint. Let $c \in\Out(\mathbb{F})$ stabilize the free factors $A$ and $C$ with $c|_A =1$ in $\Out(A)$. Finally, let $T \in \mathcal{K}^0$ and $f \in \Out(\mathbb{F})$ be arbitrary. Then $f$ induces an isomorphism $f: \mathcal{F}(A) \to \mathcal{F}(fA)$ and we have the following:
\begin{enumerate}
\item f(A) and f(B) meet and $\pi_{fA}(fB) = f(\pi_A(B)) \subset \mathcal{F}(fA)$.
\item $\pi_{fA}(fT) = f(\pi_A(T)) \subset \mathcal{F}(fA)$.
\item $\pi_{A}(cB) = \pi_A(B) \subset \mathcal{F}(A)$.
\item $\pi_{A}(cT) = \pi_A(T )\subset \mathcal{F}(A)$.
\end{enumerate}
\end{proposition}

For the applications in this paper, a slightly stronger condition than meeting is necessary on free factors $A$ and $B$. In particular, we need their meeting representatives to generate the ``correct'' subgroup of $\F$.  More precisely, say that two free factors $A$ and $B$ of $\F$ \emph{overlap} if there are representatives in their conjugacy classes, still denoted $A$ and $B$, so that $A \cap B = x \neq\{1\}$ is proper in both $A$ and $B$ and the subgroup generated by these representatives $\langle A, B \rangle\le \F$ is isomorphic to $A*_xB$. Note that the first condition here is exactly that $A$ and $B$ meet.

\begin{remark}
Suppose the free factors $[A], [B] \in \FF$ overlap and select representatives in their conjugacy classes so that $A \cap B = x$ is nontrivial and proper in both $A$ and $B$. Note that as in Lemma \ref{welldefined} the free factor $x$ is not necessarily unique up to conjugacy, but once the conjugacy class of $x$ is fixed the subgroup $H = \langle A, B \rangle$ generated by these conjugacy class representatives is itself determined up to conjugacy in $\F$. Since $A$ and $B$ overlap, $x$ can be chosen so that $H \cong A*_xB$ and it is not difficult to verify that $H$ is finitely generated and self-normalizing. So, for example, if $T \in S'$, then $\pi_A(T) = \pi_A(\pi_{\S(H)})(T)$ by Lemma \ref{properties1}. Projections of meetings factors, however, may slightly change. In particular, $A$ and $B$ are free factors of $H$ that overlap, but with conjugacy now considered in $H$, $x$ is their unique intersection up to conjugacy. In general, we use the notation $\pi_A(B \le H)$ to denote the projection of $B$ into the free factor complex of $A$ when $B$ is considered as a free factor of the free group $H$. Note that in this case $\pi_A(B \le H) = \{[x]\} \subset \pi_A(B) \subset \mathcal{F}(A)$ and so although the choice of $x$ and, hence, $H$ is not uniquely determined by the overlapping free factors $A$ and $B$, this ambiguity is not significant when considering projections.
\end{remark}

\subsection{The Bestvina-Feighn Projections } \label{BF}
In \cite{BFproj}, the authors show that there is a finite coloring of the vertices of the free factor complex $\FF_n$ and an $M \ge 0$ so that if $A$ and $B$ are free factors either having the same color or with $d_{\FF}(A,B) > 4$, then there is a well-defined projection $\pi_{\S(A)}^{\mathrm{BF}}(B) \subset \mathcal{S}(A)$ whose diameter is less than or equal to $M$. Moreover, these projections have properties that are analogous to subsurface projections. Their projection is defined by choosing any $T \in \K^0_n$ so that the marked graph $T / \F_n$ contains an embedded  subgraph whose fundamental group represents $B$ and taking $\pi_{\S(A)}^{\mathrm{BF}}(B) = \pi_{\S(A)}(T).$ It is shown that when $A$ and $B$ satisfy the stated conditions, this projection is coarsely independent of the choice of $T$. 

Free factors that meet, however, do \emph{not} satisfy the conditions stated above and it is easy to construct examples where $A$ and $B$ meet but the projection $\pi_{\S(A)}^{\mathrm{BF}}(B)$ does not have finite diameter in $\S(A)$ (as the choice of $T$ is varied). Despite this, Lemma \ref{welldefined} shows that if we further project to the free factor complex of $A$ we obtain a set with finite diameter. This shows that when the free factors $A$ and $B$ meet, the projection $\pi_A(B)$ defined in this paper agrees coarsely with the projection $\pi( \pi_{\S(A)}^{\mathrm{BF}}(B)) \subset \mathcal{F}(A)$. In Section \ref{relations}, we relate the projections discussed here with those of \cite{SS}.

\section{The homomorphisms $A(\Gamma) \to \Out(\mathbb{F}_n)$} \label{homomorphisms}
In this section, we present the most general version of our theorem.  Technical conditions are unavoidable since, unlike the surface case, free factors do not uniquely determine splittings and defining the support of an outer automorphism is more subtle. After presenting general conditions, we give a specific construction to which our theorem applies. The idea is to replace the surface in the mapping class group situation with a graph of groups decomposition of $\mathbb{F}$.

\subsection{Admissible systems}
Let $\mathcal{A} = \{A_1, \ldots, A_n \}$ be a collection of (conjugacy classes of) rank $\ge 2$ free factors of $\mathbb{F}$ such that for $i \neq j$ either
\begin{enumerate}
\item $A_i$ and $A_j$ are \emph{disjoint}, that is they are vertex groups of a common splitting, or
\item $A_i$ and $A_j$ \emph{overlap}, so in particular $\pi_{A_i}(A_j) \neq \emptyset \neq \pi_{A_j}(A_i)$.
\end{enumerate}
Then we say that $\mathcal{A}$ is an \emph{admissible collection} of free factors of $\mathbb{F}$. Let $\Gamma = \Gamma_{\mathcal{A}}$ be the coincidence graph for $\mathcal{A}$. This is the graph with a vertex $v_i$ for each $A_i$ and an edge connecting $v_i$ and $v_j$ whenever the free factors $A_i$ and $A_j$ are disjoint.

An outer automorphism $f_i \in \Out(\mathbb{F})$ is said to be \emph{supported} on the factor $A_i$ if $f_i(A_j) = A_j$ for each $v_j$ in the star of $v_i \in \Gamma^0$ and $f_i|_{A_j} = 1 \in \Out(A_j)$ for each $v_j$ in the link of $v_i \in \Gamma^0$. Informally, $f_i$ is required to stabilize and act trivially on each free factor in $\mathcal{A}$ that is disjoint from $A_i$ as well as stabilize $A_i$ itself. We say that $f_i$ is \emph{fully supported} on $A_i$ if in addition $f_i|_{A_i} \in \Out(A_i)$ is fully irreducible. Finally, we call the pair $\mathcal{S} =(\mathcal{A},\{f_i\})$ an \emph{admissible system} if the $f_i$ are fully supported on the collection of free factors $\mathcal{A}$ and for each $v_i,v_j$ joined by an edge in $\Gamma$, $f_i$ and $f_j$ commute in $\Out(\mathbb{F})$ (this condition is made unnecessary in the construction of the next section).

Given an admissible system $\mathcal{S} = (\mathcal{A},\{f_i\})$, we have the induced homomorphism
$$\phi= \phi_{\mathcal{S}}: A(\Gamma) \to \Out(\mathbb{F}_n)$$
defined by mapping $v_i \mapsto f_i .$  Our main theorem is the following:

\begin{theorem}\label{main}
Given an admissible collection $\mathcal{A}$ of free factors for $\mathbb{F}$ with coincidence graph $\Gamma$ there is a $C \ge 0$ so that if outer automorphisms $\{f_i\}$ are chosen to make $\mathcal{S} = (\mathcal{A},\{f_i\})$ an admissible system with $\ell_{A_i}(f_i) \ge C$ then the induced homomorphism $\phi= \phi_{\mathcal{S}}: A(\Gamma) \to \Out(\mathbb{F})$ is a quasi-isometric embedding.
\end{theorem}

It is worth noting that since right-angled Artin groups are torsion-free, homomorphisms from $A(\Gamma)$ that are quasi-isometric embeddings are injective. Since it requires a different set of terminology as well as constants that need to be determined, we save the statement and proof that these homomorphisms induce quasi-isometric orbit maps into Outer space for the appendix.

\subsection{Splitting contruction}
Here we present a particular type of graph of groups decompositions of $\mathbb{F}$ to which our theorem applies and use it to construct examples. Let $\mathcal{G}$ be a free splitting of $\mathbb{F}$ along with a family of collapse maps 
$$p_i: \mathcal{G} \to \mathcal{G}_i$$
to splittings $\mathcal{G}_i$, satisfying the following conditions: 
\begin{enumerate}
\item Each splitting $\mathcal{G}_i$ has a preferred vertex $v_i \in \mathcal{G}_i$ so that all edges of $\mathcal{G}_i$ are incident to $v_i$.
\item Setting $G_i = p^{-1}(v_i) \subset \mathcal{G}$ we require that for $i \neq j$ one of the two following conditions hold: either $(i)$ $G_i$ and $G_j$ are \emph{disjoint}, meaning that $G_i \cap G_j = \emptyset$, or $(ii)$ $G_i \cap G_j$ is a subgraph whose induced subgroup is nontrivial and proper in each of the subgroups induced by $G_i$ and $G_j$. In the latter case, we say the subgraphs \emph{overlap}.
\end{enumerate}

We call the splitting $\mathcal{G}$ satisfying these conditions a \emph{support graph} and note that the above data is determined by the collection of subgraphs $G_i$. For such a splitting of $\mathbb{F}$, we set $A_i = \pi_1(G_i) = (\mathcal{G}_i)_{v_i} \in \mathcal{FF}$, the vertex groups of the vertex $v_i$ in $\mathcal{G}_i$. It is clear from the above conditions that such a collection of free factors forms an admissible collection $\mathcal{A}(\mathcal{G})$ and that $\Gamma_{\mathcal{A}(\mathcal{G})}$ is precisely the coincidence graph of the $G_i$ in $\mathcal{G}$ .

Next, consider the outer automorphisms that will generate the image of our homomorphism. For each $i$, chose an $f_i \in \Out(\F_n)$ which preserves the splitting  $\mathcal{G}_i$, induces the identity automorphism on the underlying graph of $\mathcal{G}_i$, and restricts to the identity on the complement of $v_i$ in $\mathcal{G}_i$. In this case, we say that $f_i$ is \emph{supported} on $G_i$ (or $v_i$) and if the restriction of $f_i$ to the free factor $A_i$ is fully irreducible, we say that $f_i$ is \emph{fully supported} on $G_i$ (or $v_i$). With these choices, the pair $\mathcal{S}(\mathcal{G}) =(\mathcal{A}(\mathcal{G}), \{f_i\})$ is an admissible system.  
Indeed, the only condition to check is that if $v_i$ and $v_j$ represent disjoint free factors, then the outer automorphisms $f_i$ and $f_j$ commute. Observe that since $G_i$ and $G_j$ are disjoint subgraphs of $\mathcal{G}$ we may collapse each to a vertex to obtain a common refinement $\mathcal{G}_{ij}$ of $\mathcal{G}_i$ and $\mathcal{G}_j$ which has vertices with associated groups $(\mathcal{G}_i)_{v_i}$ and $(\mathcal{G}_j)_{v_j}$. Label these vertices of $\mathcal{G}_{ij}$ $v_i$ and $v_j$ corresponding to the subgraphs $G_i$ and $G_j$ of $\mathcal{G}$. From the fact that $f_i$ and $f_j$ are supported on $G_i$ and $G_j$, respectively, it follows that they both stabilize the common refinement $\mathcal{G}_{ij}$ and are each supported on distinct vertices, namely $v_i$ and $v_j$. This implies that $f_i$ and $f_j$ commute in $\Out(\F)$. Hence,  $\mathcal{S}(\mathcal{G}) =(\mathcal{A}(\mathcal{G}), \{f_i\})$ is an admissible system inducing a homomorphism
$$\phi_{\mathcal{S}(\mathcal{G})}: A(\Gamma_{\mathcal{G}}) \to \Out(\mathbb{F})$$
given by
$$v_i \mapsto f_i $$
as before.
With this setup, our main result can be restated as follows:

\begin{corollary}\label{supportgraphversion}
Suppose $\mathcal{G}$ is a free splitting of $\F$ that is a support graph with subgraphs $G_i$ for $1 \le i \le n$ having coincidence graph $\Gamma$.  There is a $C \ge 0$ so that if for each $i$, $f_i \in \Out(\F)$ is fully supported on $G_i$ with $\ell_{A_i}(f_i) \ge C$, then the induced homomorphism $\phi_{\mathcal{S}(\mathcal{G})} : A(\Gamma) \to \Out(\F)$ is a quasi-isometric embedding.
\end{corollary}

We remark that once a support graph $\mathcal{G}$ is constructed with $\pi_1 \mathcal{G} = \F_n$, there is no obstruction to finding $f_i$ fully supported on $G_i$ with large translation length on $\mathcal{F}(A_i)$. Corollary \ref{supportgraphversion} then implies that there exist homomorphisms $\phi_{\mathcal{S}(\mathcal{G})} : A(\Gamma) \to \Out(\F_n)$ which are quasi-isometric embeddings.

\subsection{Constructions and applications}
We use our main theorem to construct quasi-isometric homomorphisms into $\Out(\mathbb{F}_n)$ beginning with an arbitrary right-angled Artin group $A(\Gamma)$. We provide a bound on $n$ given a measurement of complexity of $\Gamma$.

First, it is an easy matter to use the splitting construction to start with a graph $\Gamma$ and find a quasi-isometric embedding $A(\Gamma) \to \Out(\F_n)$, with $n$ depending on $\Gamma$. We first illustrate this with an example and then give a general procedure. Note that although using the splitting construction is simple, it will always require that $n$ is rather large compared to $\Gamma$. As demonstrated in Example $2$, more creative choices of admissible system can, however, be used to reduce $n$.

\begin{example} 
Let $\Gamma = \Gamma_5$ be the pentagon graph with vertices labeled counter-clockwise $v_0, v_2, v_4, v_1, v_3$ as in Figure \ref{fig1}, and let $\Gamma^c$ be the same graph with vertices labeled cyclically $v_0, \ldots, v_4$. Take $\mathcal{G}$ to be the graph of groups with underlying graph $(\Gamma^c)'$, the barycentric subdivision of $\Gamma^c$, with trivial vertex group labels on the vertices of $\Gamma^c$ and infinite cyclic group labels on the subdivision vertices. Note that $\pi_1 \mathcal{G} = \F_6$. Set $G_i$ $(1\le i \le 4)$ equal to the subgraph of $\mathcal{G}$ consisting of the vertex labeled $v_i$, its two adjacent subdivision vertices, and the edges joining these vertices to $v_i$. Observe that $G_i$ and $G_j$ have empty intersection if and only if $v_i$ and $v_j$ are joined by an edge in $\Gamma$, and if $G_i$ and $G_j$ intersect then their intersection is a vertex with nontrivial vertex group. Hence, $\mathcal{G}$ is a support graph with subgraphs $G_i$ whose coincidence graph is $\Gamma$ and so by Corollary \ref{supportgraphversion} there is a constant $C$ such that choosing any collection of outer automorphisms $f_i$ fully supported on the collection $G_i$ with $\ell_{A_i}(f_i) \ge C$ determines a homomorphism $A(\Gamma_5) \to \Out(\F_6)$ that is a quasi-isometric embedding.  In Example $2$, we improve this construction by modifying $\mathcal{G}$.\\
\end{example}

Now fix any simplicial graph $\Gamma$ with $n$ vertices labeled $v_1, \ldots, v_n$. We give a general procedure for producing a support graph $\mathcal{G}$ with subgraphs $G_i$ whose coincidence graph is $\Gamma$. By Corollary \ref{supportgraphversion}, this allows one to construct homomorphisms $A(\Gamma) \to \Out(\pi_1(\mathcal{G}))$ which are quasi-isometric embeddings for any right-angled Artin group. First, assume that the \emph{complement graph} $\Gamma^c$ is connected. This is the subgraph of the complete graph on the vertices of $\Gamma$ with edge set given by the complement of the edge set of $\Gamma$. Let $(\Gamma^c)'$ be the barycentric subdivision of $\Gamma^c$ where we reserve labels $v_i$ for the vertices of $(\Gamma^c)'$ that are vertices of $\Gamma^c$ and label the vertex of $(\Gamma^c)'$ corresponding to the edge $(v_i,v_j)$ of $\Gamma^c$ by $v_{ij}$. Hence, in $(\Gamma^c)'$ the vertex $v_{ij}$ is valence two and is connected by an edge to both $v_i$ and $v_j$. Set $G_i$ equal to the star of the vertex $v_i$ in $(\Gamma^c)'$, i.e. $G_i$ is the union of edges incident to $v_i$ together with their vertices. Now take $\mathcal{G}$ to be the graph of groups with underlying graph $(\Gamma^c)'$  and infinite cyclic vertex group labels for each vertex $v_{ij}$, $i\neq j$. For vertices $v_i$ there are two cases for vertex groups: If $v_i$ has valence one in $\mathcal{G}$ then we label it with infinite cyclic vertex group and otherwise we give it trivial vertex group. 

With these vertex groups, $\mathcal{G}$ becomes of graph of groups decomposition for $\F_n$ which is a support graph for the collection of subgraphs $G_i$ with coincidence graph $\Gamma$. Indeed, $G_i$ and $G_j$ have nonempty intersection in $\mathcal{G}$ if and only if $v_i$ and $v_j$ are joined by and edge in $\Gamma^c$. When this is the case, their intersection is a single vertex with infinite cyclic vertex group and this vertex group is proper in each of the groups induced by $G_i$ and $G_j$. We can also calculate the rank of $\pi_1 \mathcal{G}$. By construction, the rank of $\pi_1 \mathcal{G}$ is equal to the rank of the fundamental group of the underlying graph plus the number of nontrivial vertex groups on $\mathcal{G}$. Since there is a nontrivial vertex group for each edge of $\Gamma^c$ and each vertex of $\Gamma^c$ of valence one, the rank of $\pi_1 \mathcal{G}$ equals 
$$1 +2 |\mathrm{E}(\Gamma^c)| - |\mathrm{V}(\Gamma^c)| + |\text{valence 1 vertices of }\Gamma^c|. $$
Translating this into a function of $\Gamma$, we see that the rank of $\pi_1 \mathcal{G}$ is
$$1+ |\mathrm{V}(\Gamma)| \cdot (|\mathrm{V}(\Gamma)| -2) - |\mathrm{E}(\Gamma)| + |\text{valence $n-2$ vertices of } \Gamma|, $$
and we refer to this quantity as the complexity of $\Gamma$, denoted $c(\Gamma)$.

When $\Gamma^c$ is not connected it decomposes into components $\Gamma^c = \sqcup_{i=1}^l \Delta_i$ and it is not difficult to show that $A(\Gamma) = A(\Delta_1^c) \times \ldots \times A(\Delta_l^c)$. In this case, we set $c(\Gamma) = \sum_i c(\Delta_i^c)$ and the corresponding supported graph is constructed as follows: Let $\mathcal{G}(\Delta_i^c)$ be the support graph constructed as above for the graph $\Delta_i^c$, or any support graph with coincidence graph $\Delta_i^c$. Take $\mathcal{G}$ to be the support graph built by taking the wedge of $l$ intervals (at one endpoint of each) and attaching the other endpoint of the $i$th interval to an arbitrary vertex of $\mathcal{G}(\Delta_i^c)$. The graph of groups structure on $\mathcal{G}$ is induced by that of $\mathcal{G}(\Delta_i^c)$ with a trivial group label at the wedge vertex. Then $\mathcal{G}$ is a support graph with coincidence graph $\Gamma$ and complexity $c(\Gamma)$. As noted above, the existence of a support graph with coincidence graph $\Gamma$ implies the following (the statement about Outer space is established in the appendix):

\begin{corollary}
For any simplicial graph $\Gamma$,  $A(\Gamma)$ admits a homomorphism into $\Out(\F_n)$, with $n \le c(\Gamma)$, which is a quasi-isometric embedding, and for which the orbit map to Outer space is a quasi-isometric embedding.
\end{corollary}

The next example shows how Theorem \ref{main} can be used to give quasi-isometric embeddings into $\Out(\F_n)$ for smaller $n$ than using support graphs, as in our construction above.

\begin{figure}[htbp]
\begin{center}
\includegraphics[height=50mm]{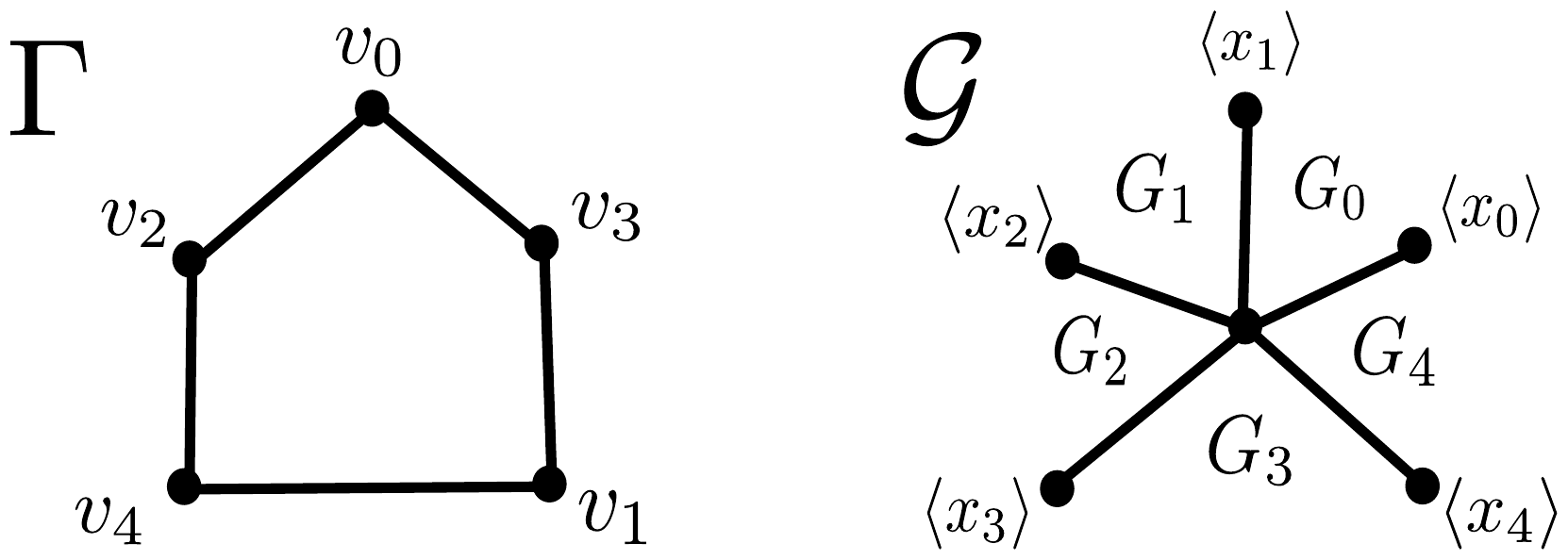}
\caption{$\mathbb{F}_5 = \pi_1(\mathcal{G})$}
\label{fig1}
\end{center}
\end{figure}

\begin{example}
Again, let $\Gamma = \Gamma_5$ be the pentagon graph with vertices labeled counter-clockwise $v_0, v_2, v_4, v_1, v_3$ as in Figure \ref{fig1}. Take $\mathcal{G}$ as in Figure \ref{fig1}. This is a graph of groups decomposition for $\F_5$; the  central vertex has trivial vertex group and the $5$ valence one vertices joined to the central vertex each have infinite cyclic vertex group, with generators labeled $x_0 \ldots, x_4$. $\mathcal{G}$ can be thought of as a ``folded'' version of the support graph that appears in Example $1$. For $0\le i \le4$, let $G_i$ be the smallest connected subgraph containing the vertices labeled $x_i$ and $x_{i+1}$, with indices taken mod $5$. Note that $\mathcal{G}$ together with the subgraphs $G_i$ is \emph{not} a support graph; for example $G_0$ and $G_2$ intersect in a vertex with trivial vertex group. Despite this, for $i = 0, \ldots ,4$, $A_i =\pi_1G_i = \langle x_i ,x_{i+1} \rangle$ does form an admissible collection of free factors with coincidence graph $\Gamma_5$. Hence, by Theorem \ref{main} there exists a $C \ge 0$ so that if there are outer automorphisms $f_i \in \Out(\F_5)$ making $(\{A_i\}, \{f_i\})$  an admissible system with $\ell_{A_i}(f_i) \ge C$ then the induced homomorphism $\phi: A(\Gamma_5) \to \Out(\F_5)$ is a quasi-isometric embedding. Choosing such a collection in this case is straightforward. Specifically, let $B_i = \langle x_{i+2}, x_{i+3}, x_{i+4} \rangle$ and choose $f_i \in \Out(\mathbb{F}_5)$ for $i= 0, \ldots, 4$ so that 
\begin{enumerate}
\item $f_i([A_i]) = [A_i]$ and $f_i([B_i]) = [B_i]$,
\item the restriction $f_i|{A_i} \in  \Out(A_i)$ is fully irreducible with $\ell_{A_i}(f_i) \ge C$, and
\item the restriction $f_i|{B_i} = 1 \in  \Out(B_i)$.
\end{enumerate}
With these choices, it is clear that each $f_i$ is fully supported on $A_i$ and that $f_i$ and $f_j$ commute if and only if $v_i$ and $v_j$ are joined by an edge of $\Gamma$. This makes $\mathcal{S} =(\{A_i\}, \{f_i\})$ into an admissible system with $\ell_{A_i}(f_i)\ge C$ and so the induced homomorphism
$$\phi_{\mathcal{S}}: A(\Gamma_5) \to \Out(\F_5) $$
is a quasi-isometric embedding. In fact, as we shall see in the proof of the main theorem, the required translation length is simple to determine. Further, as each of free factors $A_i$ in the admissible system is rank $2$, the free factor complex $\mathcal{F}(A_i)$ is the Farey graph where translation lengths can be computed.

As an application, it is well-known that $A(\Gamma_5)$ contains quasi-isometrically embedded copies of $\pi_1(\Sigma_2)$, the fundamental group of the closed genus $2$ surface (see \cite{CrispWiest}). Restricting the homomorphism constructed above to such a subgroup we obtain quasi-isometric embeddings 
$$\pi_1(\Sigma_2) \to \Out(\F_5). $$
As we will see in the appendix, these homomorphisms can be chosen to give quasi-isometric orbit maps into Outer space, $\mathcal{X}_5$.
\end{example}

\begin{example} 
Let $\F_3 = \langle a,b,c \rangle$ and set $A_i = \langle a,b^ic \rangle$ for $i \ge 0$. It is not difficult to verify that for $i \neq j$, $\pi_{A_i}(A_j) = \{[a]\}$ and $\langle A_i, A_j \rangle = \langle a,b^{j-i}, b^{j}c \rangle \cong A_i *_{\langle a \rangle} A_j$. See Stallings's paper \cite{st83} for how to efficiently compute such intersections. Hence, for $N \ge 1$, the collection $\mathcal{A}_N = \{A_i :  0\le i \le N\}$ is an admissible collection of pairwise overlapping free factors and so there are $C_N$ such that choosing any collection of outer automorphisms $\{f_i : 0 \le i \le N\}$ fully supported on the collection $\mathcal{A}_N$ with $\ell_{A_i}(f_i) \ge C_N$ determines a homomorphism $\phi_N: \F_N \to \Out(\F_3)$ that is a quasi-isometric embedding.  In fact, we will see in the proof of Theorem \ref{main} that we may take $C = C_N$ to be constant over all $N$ and obtain a uniform lower bound on $\Out(\F_3)$-word length of $\phi_N(x)$ in terms of word length of $x$ in $\F_N$, independent of $N$.
\end{example}

\section{Splittings and submanifolds}
It is important to have a topological interpretation of our projections in order to prove the version of Behrstock's inequality that appears in the next section. We first review some facts about embedded surfaces in $3$-manifolds and the splittings they induce.

\subsection{Surfaces and splittings}\label{surfacesandsplittings}
It is well-known that codimension $1$ submanifolds induce splittings of the ambient manifold group \cite{Shalenrep}. We review some details here, focusing on the case when then inclusion map is not necessarily $\pi_1$-injective.

For our application, begin with an orientable, connected $3$-manifold $X$ possibly with boundary and a property embedded, orientable surface $F$. We do not require that $F$ is connected or that each component of $F$ is $\pi_1$-injective. Working, for example, in the smooth setting, choose a tubular neighborhood  $N \cong F \times I$ of $F$ in $X$ whose restriction $N \cap \partial X$ is a tubular neighborhood of the boundary of $F$ in $\partial X$. Let $G$ denote the graph \emph{dual} to $F$ in $X$. This is the graph with a vertex for each component of $X \setminus \mathrm{int}(N)$ and an edge $e_f$ for each component $f \subset F$ that joins vertices corresponding to the (not necessarily distinct) components on either side of $f$. We may consider $G$ as embedded in $X$ and, after choosing an appropriate embedding, $G$ is easily seen to be a retract of $X$. The retraction is obtained by collapsing each complementary component of $N$ to its corresponding vertex and projecting $f \times I$ to $I$ for each component $f$ of $F$. Here, $I$ is the closed interval $[-1,1]$ and $f \times \{0\}$ corresponds  under the identification $N \cong F \times I$ to $f \subset N$.

Let $\tilde{X}$ denote the universal cover of $X$ and let $\tilde{N}$ and $\tilde{F}$ denote the complete preimage of $N$ and $F$, respectively. Let $T_F$ denote the graph dual to $\tilde{F}$ in $\tilde{X}$. Since $T_F$ is a retract of the connected and simply connected space $\tilde{X}$, $T_F$ is a tree. We call $T_F$ the \emph{dual tree} to the surface $F$ in $X$. As $\tilde{F}$ and $\tilde{X} \setminus \tilde{N}$ are permuted by the action of $\pi_1(X)$, we obtain a simplicial action $\pi_1(X) \curvearrowright T_F$, up to the usual ambiguity of choosing basepoints. The following is an easy exercise in covering space theory; it appears in \cite{Shalenrep}.

\begin{proposition}\label{stabilizers}
With the above notation, let $v$ be a vertex of $T_F$ corresponding to a lift of a component $C \subset X \setminus N$ and $e$ an edge of $T_F$ corresponding to a lift of a component $f \subset F$. Then
\begin{enumerate}
\item stab$(v) = \mathrm{im}(\pi_1 C \to \pi_1X)$
\item stab$(e) = \mathrm{im}(\pi_1 f \to \pi_1X)$
\end{enumerate}
where both equalities are up to conjugation in $\pi_1X$.
\end{proposition}

The action $\pi_1X \curvearrowright T_F$ provides a splitting of $\pi_1X$ via Basse-Serre theory. The corresponding graph of groups decomposition of $\pi_1X$ has underlying graph $G = T_F / \pi_1X$, the graph dual to $F \subset X$, with vertex and edge groups as given in Proposition \ref{stabilizers}. A subgraph $G' \subset G$ \emph{carries} a subgroup $H \le \pi_1X$ if the subgroup induced by the subgraph $G'$ contains $H$, up to conjugacy. 

Now specialize to the situation where the action $\pi_1X \curvearrowright T_F$ has trivial edge stabilizers. The following proposition determines when the dual tree to a surface is minimal. First, say that a connected component $f \subset F$ is \emph{superfluous} if $f$ separates $X$ and to one side bounds a relatively simply connected submanifold, i.e. $X \setminus f = X_1 \sqcup X_2$ and $\text{im}(\pi_1(X_1) \to \pi_1(X)) = 1$. A component of $F$ that is not superfluous is said to \emph{split} $X$. Also, use the notation $T^{\mathrm{min}}$ to denote the unique minimal subtree associated to an action on the tree $T$, see Section \ref{basics}.

\begin{proposition}\label{minimal}
Let $F$ be an orientable, properly embedded surface in the orientable $3$-manifold $X$ with $\mathrm{im}(\pi_1 f \to \pi_1X) = 1$ for each components $f$ of $F$. Then the edge $e_f \subset T$ corresponding to a lift of the component $f \subset F$ is contained in the minimal subtree $T_F^{\mathrm{min}}$ if and only if $f$ splits $X$.

\end{proposition}

\begin{proof}
First suppose that the edge $e_f$ whose orbit corresponds to the lifts of $f$ is not in the minimal subtree $T^{\mathrm{min}}$. Setting $G = T_F / \pi_1X$ and $G^{\mathrm{min}} = T_F^{\mathrm{min}} / \pi_1X$, the image of $e_f$ in $G$ does not lie in $G^{\mathrm{min}}$. Since $G^{\mathrm{min}}$ carries the fundamental group of $X$, the image of $e_f$ in $G$ must separate and the component of its complement not containing $G^{\mathrm{min}}$ has all trivial vertex groups. In $X$, this implies that the component $f \subset F$ separates $X$ and to one side bounds a component whose fundamental group when included into $\pi_1X$ is trivial. Hence, $f$ is superfluous.

Now suppose that $f$ is a component of $F$ that is superfluous. Then $f$ corresponds to a separating edge $e$ in $G = T_F / \pi_1 X$, with lift $e_f \subset T_F$, whose complement in $G$ contains a component with trivial induced subgroup. Hence, this component of $G \setminus e$ is a tree with trivial vertex groups. Set $G'$ equal to the other component of the complement of $e$ in $G$. Then $G'$ carries all of $\pi_1X$ and so its complete preimage in $T_F$ is connected, $\pi_1X$ invariant, and does not contain the edge $e_f$. Hence, $e_f$ is not in $T^{\mathrm{min}}$.
\end{proof}

We will use the above proposition in the following manner: If $f \subset F$ splits $X$ then $T_f$ is a $1$-edge collapse of $T_F$ corresponding to a $1$-edge splitting of $\pi_1X$.

\subsection{Topological projections}
The purpose of this section is to give a topological description of the projections $\pi_A(T)$ in terms of submanifolds of the manifold $M_n$. This will allow us to prove a version of Behrstock's inequality in the next section. As discussed below, these are similar to the submanifold projections of \cite{SS} and this section serves to explain the connection between these projections and the ones of \cite{BFproj}. To verify that our description is accurate, we rely on Hatcher's normal position for spheres in $M = M_n$ and its generalization in \cite{HOD}.  Let $\tilde{M}$ denote the universal covers of $M$. We say that essential sphere systems $S_1$ and $S_2$ in $M$ are in \emph{normal position} if for $\tilde{S_1}$ and $\tilde{S}_2$, the complete preimage of $S_1$ and $S_2$ in $\tilde{M}$, any spheres $s_1\in \tilde{S}_1$ and $s_2 \in \tilde{S}_2$ satisfy each of the following:
\begin{enumerate}
\item $s_1$ and $s_2$ intersect in at most one component and
\item no component of $s_1 \setminus s_2$ is a disk that is isotopic relative its boundary to a disk in $s_2$.
\end{enumerate}
This definition is easily seen to be equivalent to Hatcher's original notion of normal position in the case where one of the sphere systems is maximal \cite{Hatcher}. In particular, the authors in \cite{HOD}  use Hatcher's original proof of existence and uniqueness of normal position to show the following:

 \begin{lemma} \label{normalp}
Any two essential sphere systems $S_0$ and $S$ can be isotoped to be in normal position. Also, normal position is unique in the following sense:
Let $S_0$ be a sphere system of $M_n$, and let $S,S'$ be two isotopic spheres in $M_n$ which are in normal position with respect to $S_0$. Then there is a homotopy between $S$ and $S'$ which restricts to an isotopy on $S_0$.
 \end{lemma}
 
Fix sphere systems $S$ and $S_A$ and a preferred component $C_A \subset M \setminus S_A$. In what follows we assume that $S_A = \partial C_A$. When this is the case, we refer to $C_A$ as the \emph{splitting component} and observe that such a component is homeomorphic to $M_{k,s}$. Let $A$ be the free factor defined up to conjugacy by $\pi_{1} (C_A)$ and let $T = T_S$ be the free splitting of $\mathbb{F}$ determined by the sphere system $S$. Since we are interested the projection of the splitting $\mathbb{F} \curvearrowright T$ to the free splitting complex of $A$, our aim is a topological interpretation of the projection $\pi_A (T) = [A \curvearrowright T^A]$.

Put $S$ and $S_A$ in normal position and consider the system of surfaces $F = S \cap C_A$. This family of surfaces is well-defined up to homotopy in $C_A$ that restricts to isotopy on $S_A$ by Lemma \ref{normalp}. Consider the graph of spaces decomposition of $C_A$ given by $F$ with dual tree $T_F$, see Section \ref{surfacesandsplittings}. Recall that a connected component $f \subset F$ is \emph{superfluous} if $f$ separates $C_A$ and to one side bounds a relatively simply connected submanifold, that is $C_A \setminus f = C_1 \sqcup C_2$ and $\text{im}(\pi_1(C_1) \to \pi_1(C_A)) = 1$. A component of $F$ that is not superfluous is said to \emph{split} $C_A$. Set $\bar{F}$ equal to $F$ minus its superfluous components and $T_{\bar{F}}$ its dual tree (that is the tree dual to the complete preimage of $\bar{F}$ in the universal cover of $C_A$). 

We claim the following about the associated splitting of $\pi_1C_A =A$: 
\begin{enumerate}
\item there is an $A$-equivariant simplicial embedding $\chi: T_F \to T$ whose image contains $T^A$,
\item an edge $e$ of $T_{F}$ maps to an edge in $T^A$ if and only if $e$ corresponds to the lift of a component of $f \subset F$ that splits $C_A$, and 
\item the projection $\pi_A(T) = A \curvearrowright T^A$ is conjugate to the $A$-tree $T_{\bar{F}}$.
\end{enumerate}

\begin{figure}
\begin{center}
\begin{large}
$\begin{tikzcd}
T_{F} \arrow[bend left]{rrd}{\chi}  \\
\tilde{C}_{A} \arrow{u} \arrow[hookrightarrow]{r} \arrow{d}{\pi|_{\tilde{C}_{A}}} & \tilde{M} \arrow{r}{\tilde{p}}\arrow{d}{\pi} &T\arrow{d} \arrow[hookleftarrow]{r}  & T^{A} \arrow{d}       \\
C_{A} \arrow[hookrightarrow]{r}    &  M \arrow{r}        &T / \F_n & T^{A}/A \arrow{l}
\end{tikzcd}$
\end{large}
\caption{Defining the map $\chi: T_F \to T$}
\label{TF}
\end{center}
\end{figure}

To prove the above claim we refer to Figure \ref{TF}, where as above the free splitting $\F\curvearrowright T$ corresponds to the sphere system $S \subset M$. Let $\pi: \tilde{M} \to M$ be the universal cover of $M$ and let $\tilde{S}$ be the complete preimage of $S$ in $\tilde{M}$. The map labeled $\tilde{p}$ is the equivariant map from $\tilde{M}$ to the tree $T$ obtained by retracting $\tilde{M}$ to the tree dual to $\tilde{S} \subset \tilde{M}$, as explained in Section \ref{surfacesandsplittings}. Hence, if we let $m$ denote the set of midpoints of edges of $T$ then $\tilde{S} = \tilde{p}^{-1}(m)$. Setting $F = S \cap C_A$ as above, we note that if $\tilde{C}_A$ is a fixed component of the preimage of $C_A$ in $\tilde{M}$ then $\pi|_{\tilde{C}_A} : \tilde{C}_A \to C_A$ is the universal cover and $\tilde{F} = (\pi|_{\tilde{C}_A})^{-1}(F) = \tilde{S} \cap \tilde{C}_A$. Hence, by definition of the dual tree to $F$ in $C_A$,  $T_F$ is precisely the tree dual to $\tilde{S} \cap \tilde{C}_A$ in $\tilde{C}_A$.

Because $\tilde{p}$ is $\mathbb{F}$-equivariant, $T' = \tilde{p}(\tilde{C}_A)$ is an $A$-invariant subtree of $T$ and so contains $T^A$, the minimal $A$-subtree of $T$. Note that by carefully choosing the projection $\tilde{p}$, we may assume that $T'$ is a subcomplex of $T$. We first show that the $A$-tree $T'$ is conjugate to the $A$-tree $T_F$. Since $T$ is dual to $\tilde{S}$ in $\tilde{M}$ and $T_F$ is dual to $\tilde{F} = \tilde{S} \cap \tilde{C}_A$ in $\tilde{C}_A \subset \tilde{M}$ each complementary component of $\tilde{F}$ in $\tilde{C}_A$ corresponds to a complementary component of $\tilde{S}$ in $\tilde{M}$. This induces a map from the vertices of $T_F$ to those of $T$ and as components of $\tilde{F}$ are contained in components of $\tilde{S}$, this map extends to a simplicial map of $A$-trees $\chi: T_F \to T$ mapping edges to edges with image $T'$. We show that this map does not fold edges and is, therefore, an immersion. This suffices to prove that $\chi: T_F \to T'$ is an $A$-conjugacy.

\begin{figure}[htbp]
\begin{center}
\includegraphics[height=50mm]{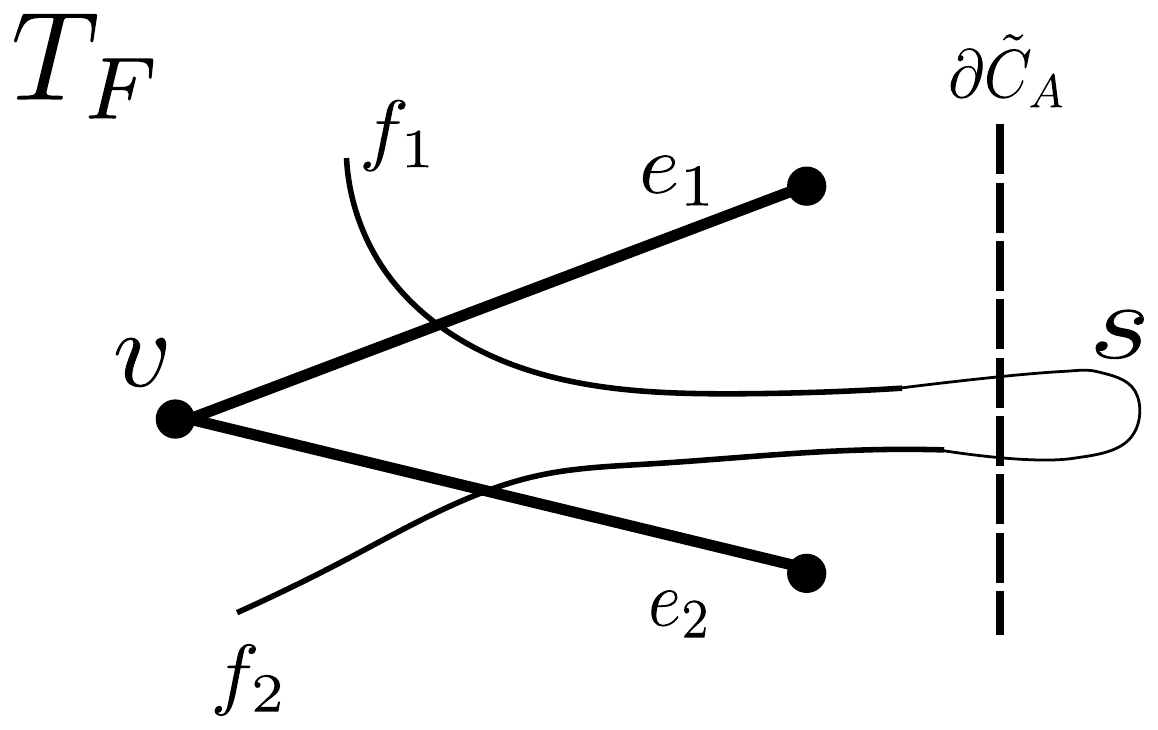}
\caption{Folding edges}
\label{fig2}
\end{center}
\end{figure}

To see that $\chi$ does not fold edge, suppose to the contrary that two edges $e_1$ and $e_2$ with common initial vertex $v$ are identified by $\chi$ (Figure \ref{fig2}). Then the edge $e_i$ is dual to a component $f_i\subset \tilde{F}$ in $\tilde{C}_A$ and these components are disjoint. Since $e_1$ and $e_2$ are folded by $\chi$, their common image $e$ in $T$ corresponds to a sphere $s \subset \tilde{S}$ which must contain $f_1$ and $f_2$ as subsurfaces. Let $\tau$ be an arc in $s$ that connects the interiors of $f_1$ and $f_2$ and intersects only the components of $s \cap \partial \tilde{C}_A$ that separate $f_1$ and $f_2$. Since each component of $\partial \tilde{C}_A$ separates $\tilde{M}$, as do all essential spheres in $\tilde{M}$, the first and last components of $\partial \tilde{C}_A$ intersected by $\tau$ must be the same. This implies that $f_1$ and $f_2$ each have boundary components on the same component of $\partial \tilde{C}_A$. Hence, the sphere $s$ intersects the same components of $\partial \tilde{C}_A$ in at least $2$ circles. This, however, contradicts normal position of the sphere systems $S$ and $\partial C_A$. We conclude that the $A$-tree $T_F$ and $T'$ are simplicially conjugate. This proves claim $(1)$ and justifies identifying $T_F$ and $T'$ through $\chi$. Observe that since $T'$ contains $T^A$, we get the induces $A$-conjugacy $\chi: T^A_F \to T^A$ on minimal subtrees.

It remains to show that $T_{\bar{F}} = T^{A}_F$, as this identifies edges of $T$ corresponding to components of $F$ that split $C_A$ with those contained in $T^A$. Since the components of $\bar{F}$ are precisely those that split $C_A$, Proposition \ref{minimal} implies that the minimal $A$-subtree of $T_F$ is $T_{\bar{F}}$ and so $T_{\bar{F}} = T^{A}_F$, as required. This completes the proofs of claims $(2)$ and $(3)$.

To summarize the above discussion:
\begin{proposition} \label{top}
Let $T \in \mathcal{S}_n'$ be a free splitting of $\mathbb{F}_n$ corresponding to the sphere system $S \subset M_n$. Fix a submanifold $C_A$, as above, with $\pi_1(C_A) = A$ and $\partial C_A$ and $S$  in normal position. Then, if one removes from $F = S \cap C_A$ the components that separate and bound relatively simply connected components, the resulting splitting $T_{\bar{F}}$ is conjugate as an $A$-tree to $\pi_{\mathcal{S}(A)}(T)$.
\end{proposition}

\subsection{Relations between the various projections}\label{relations}
In \cite{SS}, Sabalka and Savchuk define projections from the sphere complex $\S(M_n)$ to the sphere and disk complex of certain submanifolds of $M_n$. Their projections can be interpreted within the framework developed in this section, providing a simple relationship with $\pi_{\S(A)}(T)$. This answers a question asked in \cite{SS, BFproj}. However, it is important to note that as demonstrated below it is possible for each of the projections to be defined in situations when the other is not. Also, it is not clear whether the distances in the target complexes of the two projections are comparable. This section is not necessary for the rest of the paper.

Let $X$ denote a connected submanifold of $M_n$ that is a component of the complement of some sphere system; we above referred to such submanifolds as splitting components. Note that $X$ is homeomorphic to $M_{k,s}$ for some $k<n$ and $s>0$. The disk and sphere complex of $X$, denoted $\mathcal{DS}(X)$, is defined to be the simplicial complex whose vertices are isotopy classes of essential spheres and essential properly embedded disks in $X$ with $k+1$ vertices spanning a $k$-simplex when disks and spheres representing these vertices can be realized disjointly in $X$. Sabalka and Savchuk define their projections as follows: Let $S$ be an essential sphere system in $M_n$. Put $S$ and $\partial X$ in normal position and set $F =S \cap X$. The projection $\pi_X^{\mathrm{SS}}([S]) \subset \mathcal{DS}(X)$ is then defined to be the components of $F$ which are either spheres or disks. If there are no such components of $F$ then the projection is left undefined.

Fix a submanifold $X$, with $\pi_1X \neq  1$, so that $A = \pi_1X$ is a free factor of $\pi_1M_n =\F_n$. There is a \emph{partially} defined map $\Phi: \mathcal{DS}^0(X) \to \S^0(A)$ defined by taking $D \in \mathcal{DS}(X)$ and mapping it to the $A$-tree  $T_D$ if $D$ splits $X$ and leaving $\Phi(D)$ undefined otherwise. Recall that as in Section \ref{surfacesandsplittings}, $T_D$ is the dual tree to $D\subset X$. Note that this map will be defined on all vertices of $\mathcal{DS}(X)$ \emph{only} when $X$ is homeomorphic to  $M_{k,1}$. When $D$ and $D'$ are adjacent in $\mathcal{DS}(X)$ and both $\Phi(D)$ and $\Phi(D')$ are defined, then it is clear that $d_{\S(A)}(\Phi(D),\Phi(D')) \le 1$.  With this setup, we can show the following:

\begin{proposition}\label{relation}
Let $T$ be a free splitting of $\F_n$ and $S$ its corresponding sphere system in $M_n$. Let $X$ be a submanifold of $M_n$ with $\pi_1 X = A \neq 1$. If the composition $\Phi \circ \pi_X^{\mathrm{SS}}(S)$ is defined, then it is a free splitting of $A$ that has $\pi_{\S(A)}(T)$ as a refinement.
\end{proposition}

\begin{proof}
By Proposition \ref{top}, if $S$ and $\partial X$ are in normal position and $F  = S \cap X$ then $T^A$ is conjugate to $T_{\bar{F}}$ where $\bar{F}$ is the union of connected components of $F$ that split $X$. By definition, $\Phi \circ \pi_X^{\mathrm{SS}}(S)$ is the tree dual to the collection of disks $D \subset F$ that split $X$, which is nonempty by assumption. Since $D \subset \bar{F}$, the induced map $T_{\bar{F}} \to T_D$ is a collapse map. Hence, $T_{\bar{F}}$ refines $\Phi \circ \pi_X^{\mathrm{SS}}(S)$.
\end{proof}

This proposition also gives the connection between the projections of \cite{SS} and those of \cite{BFproj}. Recall that the projection $\pi_{\S(A)}^{\mathrm{BF}}(B)$ is well-defined, i.e. has bounded diameter image, when either $(1)$ $A$ and $B$ have the same color in the finite coloring of the vertices of $\FF_n$ or $(2)$ $d_{\FF}(A,B) > 4$. See \cite{BFproj} for definition of the coloring and further details.

\begin{corollary}
Let $A,B$ be free factors of $\F_n$ satisfying one of the above conditions so that the projection $\pi_{\S(A)}^{\mathrm{BF}}(B)$ is well-defined. Let $X$ be a submanifold of $M$ with $\pi_1X = A$ and let $S$ be any sphere system that contains a sphere system $S' \subset S$ whose dual tree $T_{S'}$ has $B$ as a vertex stabilizer. If the composition $\Phi \circ \pi_X^{\mathrm{SS}}(S)$ is defined, then it has bounded distance from $\pi_{\S(A)}^{\mathrm{BF}}(B)$ in $\S(A)$, where the bound depends only on $n$.
\end{corollary}

It is important to note that whether $\Phi \circ \pi_X^{\mathrm{SS}}(S)$ is defined is highly dependent on the choice of $X$ and $S$. This is demonstrated in the examples below.

\begin{proof}
By definition, we may take $\pi_{\S(A)}^{\mathrm{BF}}(B) = \pi_{\S(A)}(T_{S'})$. By Lemma \ref{refine}, this is refined by the projection $\pi_{\S(A)}(T_S)$,  and by Proposition \ref{relation}, $\pi_{\S(A)}(T_S)$ also refines $\Phi \circ \pi_X^{\mathrm{SS}}(S)$. This completes the proof since we may take as our bound the diameter of the Bestvina-Feighn projections plus $2$.
\end{proof}

We end this section with some examples that illustrate cases when one of the projections is defined and the other is not. The general idea is that while the Bestvina-Feighn projections are robust, i.e. they do not depend on how a factor is complemented, the Sabalka and Savchuk projections are highly sensitive to the submanifold that is chosen to represent a free factor.

\begin{example}
Take $M = M_4$ and $S = S_1 \cup S_2$ to be a union of two essential spheres so that $X = M \setminus S$ connected with $\pi_1X = A$. Let $f \in \Out(\F_n)$ with $f(A) = A$ but $f$ has no power that fixes $S$ in $\S(M)$. Then $\pi_{\S(A)}(f^nT_S) = \pi_{\S(A)}(T_S)$ is undefined, as $A$ fixes a vertex of $T_S$, but $\pi_X^{\mathrm{SS}}(f^nS)$ is defined for all $n\ge1$ by construction. Hence, it must be the case that each disk of  $\pi_X^{\mathrm{SS}}(f^nS)$ does not split $X$. Informally, each disk of  $\pi_X^{\mathrm{SS}}(f^nS)$ ($n \ge 1$) simply encloses some boundary components of $X$ without splitting $\pi_1X$.
\end{example}

\begin{example}
Take $M,X,A$ as above and refer to Figure \ref{SSexp} where $M$ is drawn as a handlebody and spheres are drawn as properly embedded disks; doubling the picture gives an illustration of what is described. Let $S_3$ be any sphere that separates $M$ into two components, one of which contains $S = \partial X$ and the other, denoted $Y$, has $\pi_1Y =A$. Let $R$ be the essential sphere shown in Figure \ref{SSexp} with dual tree $T_R$; R is in normal position with $S_3$. Note that $R$ splits $Y$ with non-trivial projection $\pi_{\S(A)}(T_R)$, but $Y \cap R$ has no disks of intersection and so $\pi_{Y}^{\mathrm{SS}}(R)$ is undefined. If, however, we use the submanifold $X$ to represent the free factor $A$, we see that $\pi_{X}^{\mathrm{SS}}(R)$ is the sphere $R \subset X$ and $\Phi \circ \pi_X^{\mathrm{SS}}(R) = \pi_{\S(A)}(T_R)$. 

Even if we only use the submanifold $X$, which \emph{exhausts} $M$ in the terminology of \cite{SS}, to represent the free factor $A$, the question of whether the composition $\Phi \circ \pi_X^{\mathrm{SS}}$ is defined still depends on the choice of sphere that is projected. This is because the existence of a disk in $\pi_X^{\mathrm{SS}}(R)$ that splits $X$ is highly depended on $R$ itself. In fact, it is not difficult to show the following: for any nonseperating sphere $R \subset X$ there is a $f \in \Out(\F_4)$ with $f(A) = A$ and $f|_A  = 1 \in \Out(A)$, so in particular $\pi_A(fT_R) = \pi_A(T_R) =\Phi \circ  \pi_X^{\mathrm{SS}}(R) $, but $\Phi \circ  \pi_X^{\mathrm{SS}}(fR)$ is undefined. This implies that all disks of $\pi_X^{\mathrm{SS}}(fR)$ are superfluous even though $\pi_A(fT_R) = \pi_A(T_R)$. \\
\end{example}

\begin{figure}[htbp]
\begin{center}
\includegraphics[height=50mm]{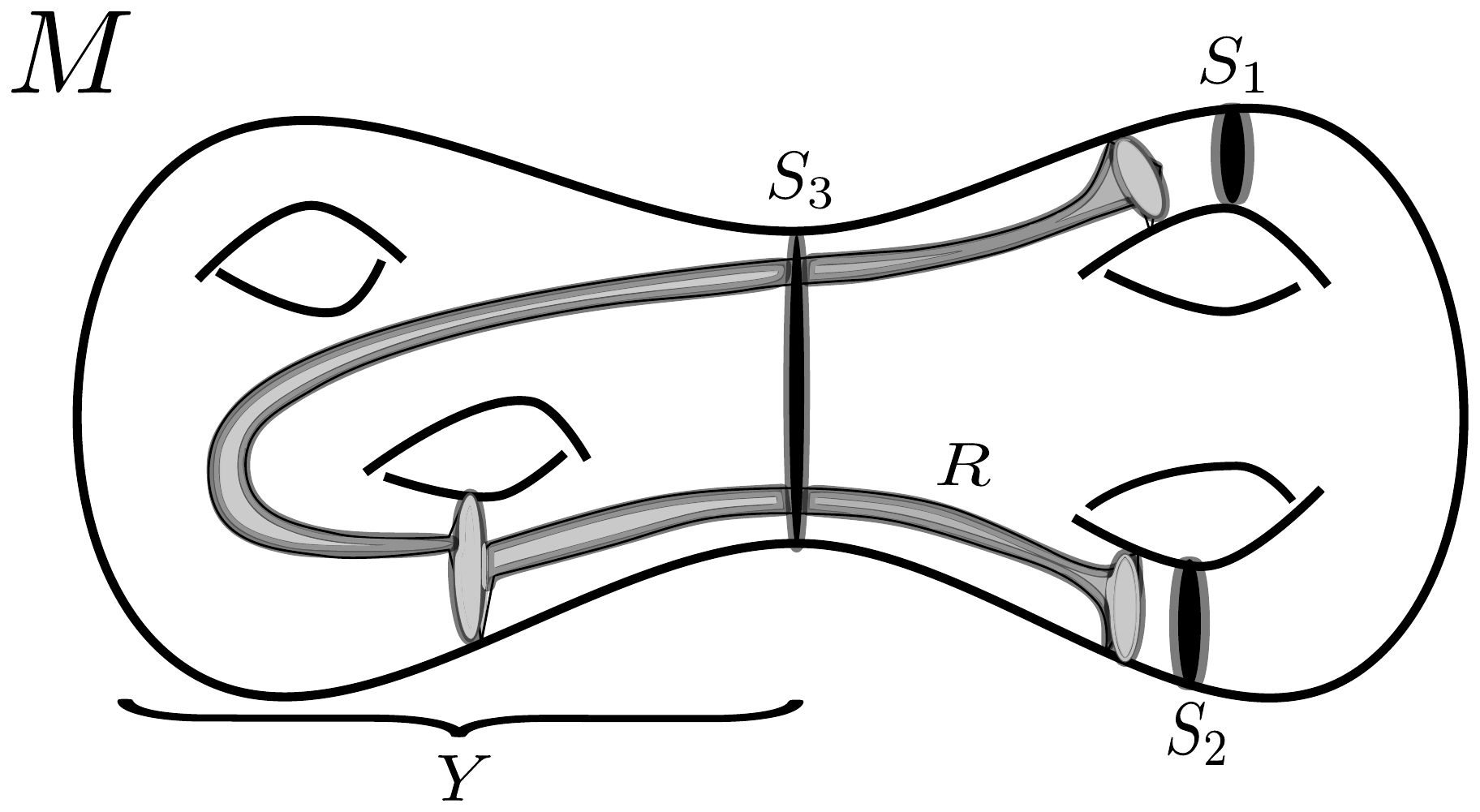}
\caption{Projecting $R$ to $Y$}
\label{SSexp}
\end{center}
\end{figure}

\section{Behrstock's Inequality}
We now introduce an analog of Behrstock's inequality for projections to the free factor complex of a free factor. For the original statement and proof in the case of subsurface projections from the curve complex of a surfaces see \cite{Be}. The proof of the free group version we give in Proposition \ref{B} is, however, more similar in spirit to the proof of the original version of Behrstock's inequality that is recorded in \cite{Ma} where it is attributed to Chris Leininger; both proofs investigate intersections of submanifolds and give explicit bounds on the distances of the projections considered.

\begin{proposition} \label{B}
There is an $M \ge 0$ so that if  $A$ and $B$ are free factors of $\mathbb{F}$ of rank $\ge 2$ that overlap, then for any $T \in \S'$ with $\pi_A(T) \neq \emptyset \neq \pi_B(T)$ we have
$$\mathrm{min}\{d_A(B,T), d_B(A,T) \} \le M .$$
\end{proposition}
\begin{proof}
Fix $T \in \S'$ that has nontrivial projection to both the free factors complex of $A$ and the free factor complex of $B$. Since $A$ and $B$ overlap we may, as in Section \ref{factorproj}, choose conjugates (still denoted $A$ and $B$) so that $A \cap B = x$, where $x \neq \{1\}$ is a proper free factor of $A$ and $B$.  Write $A = A' * x$ and $B= B'*x$ so that 
 $$H =\langle A,B \rangle \cong  A*_xB \cong A'*x*B'.$$ 
Since $\pi_A(B \le H) = \{[x]\} \subset \pi_A(B)$ in $\mathcal{F}(A)$ and $\pi_B(A \le H) = \{[x]\} \subset \pi_B(A)$ in $ \mathcal{F}(B)$ and by Lemma \ref{properties2}, $\pi_A(\pi_{\S(H)}(T)) = \pi_A(T)$ and $\pi_B(\pi_{\S(H)}(T)) = \pi_B(T)$, we have
 \begin{eqnarray*}
 d_A(B,T) &\le& d_A(\pi_A(B \le H),\pi_{\S(H)}(T)) +\mathrm{diam}_A(\pi_A(B)) \\
 &\le& d_A(x,\pi_{\S(H)}(T)) +4
\end{eqnarray*}
and similarly
\begin{eqnarray*}
 d_B(A,T) &\le& d_B(\pi_B(A \le H),\pi_{\S(H)}(T)) +\mathrm{diam}_B(\pi_B(A)) \\
 &\le& d_B(x,\pi_{\S(H)}(T)) +4
\end{eqnarray*}
Hence, it suffices to show that for $T \in \S'$ with $\pi_A(T) \neq \emptyset \neq \pi_B(T)$
$$\mathrm{min}\{d_A(x,\pi_{\S(H)}(T)), d_B(x,\pi_{\S(H)}(T)) \} \le M-4 ,$$
where $H$ is fixed as above. 

To transition to the topological picture, suppose that rank$(H) =k$ and set $M = M_k$ with a fixed identification $\pi_1M = H$. Let $S_A,S_B$ be two disjoint spheres in $M$ that correspond to the splitting $H = A' *x*B'$ via Proposition \ref{AS}. Take $C_A$ to be the submanifold with boundary $S_A$ and $\pi_1C_A = A$ and take $C_B$ to be the submanifold with boundary $S_B$ and $\pi_1C_B = B$. By construction $S_A \subset C_B$ and $S_B \subset C_A$ and so, in particular, $\partial C_A$ induces a splitting of $B = \pi_1C_B$ whose projection to $\mathcal{F}(B)$ contains $\pi_B(A \le H) = \{[x]\}$. Similarly, $\partial C_B$ induces a splitting of $A = \pi_1C_A$ whose projection to $\mathcal{F}(A)$ contains $\pi_A(B \le H) = \{[x]\}$.

Now choose any tree $T  \in \S'(H)$ with nontrivial projections to $\mathcal{F}(A)$ and $\mathcal{F}(B)$ and let $S$ be the corresponding sphere system in $M$. Put $S$ and $S_A \cup S_B$ in normal position and recall that by Proposition \ref{top}, $\pi_{\mathcal{S}(A)}(T)$ is given by the collection of components of $C_A \cap S$ that split $C_A$. With this set-up, we show that 
$$\text{min}\{d_A(\partial C_B , S), d_B(\partial C_ A, S) \} \le 12 $$
where for any sphere system $R$ in $M$, $\pi_A(R)$ denotes $\pi_A(T_R)$.

Suppose, toward a contraction, that both $d_B(\partial C_A, S)$ and $d_A(\partial C_B, S)$ are greater than $12$ and consider the forest $G$ on $S$ that is dual to the circles of intersection $\partial C_A \cap S$ and $\partial C_B \cap S$. We label the edges of $G$ dual to circles of $\partial C_A \cap S$ with ``$a$'' and those dual to $\partial C_B \cap S$ with ``$b$''. Label the vertices  of $G$ that represent components $S \setminus (\partial C_A \cup \partial C_B)$ contained in $C_A \cap C_B$ with ``$AB$'', those in $C_A$ but not $C_B$ with ``$A$'', and those in $C_B$ but not $C_A$ with ``$B$''.

Call a subtree of $G$ \emph{terminal} if it has a unique vertex that separates it from its complement in $G$. We say a subtree is an $a$- tree (or $b$- tree) if all of its edges are labeled $a$ (or $b$). 

\begin{claim}\label{notrees}
No $AB$-vertex which is the boundary of both an $a$-edge and a $b$-edge is a vertex for either a terminal $a$-tree or a terminal $b$-tree.
\end{claim}

\begin{figure}[htbp]
\begin{center}
\includegraphics[height=50mm]{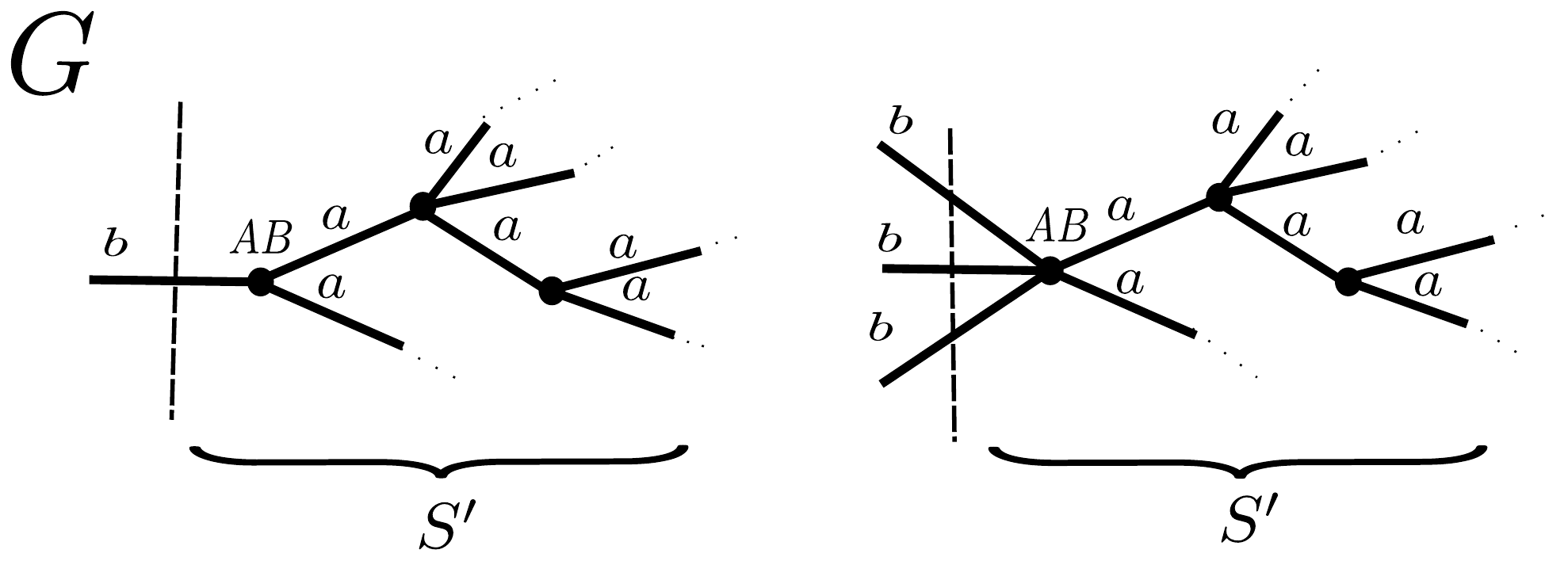}
\caption{Two cases for $S'$}
\label{fig3}
\end{center}
\end{figure}

\begin{proof} [Proof of claim \ref{notrees}]
We prove the claim for terminal $a$-trees. The proof for $b$-trees is obtained by switching the symbols $a$ and $b$.

Suppose that there is an $AB$-vertex $v$ of $G$ which bounds both an $a$-edge and a $b$-edge and is the vertex for a terminal $a$-tree. Observe that the component $S'$ of $S \cap C_B$ that corresponds to the union of $b$-edges at $v$ splits $C_B$ and so it can be used for the projection $\pi_{B}(S)$ (see the remark following Proposition \ref{minimal}). To see that $S'$ splits $C_B$, recall that if this were not the case then $C_B \setminus S' = C_1 \cup C_2$, where $C_1$ is relatively simply connected in $M$. As $S'$ contains a disk of intersection with either $C_B \cap C_A$ or $C_B \setminus C_A$ coming from a valence one vertex of the terminal $a$-tree, this disk cobounds a region $R$ contained in $C_1$ with a disk of $\partial C_A$. This shows that $R$ is relatively simply connected with sphere boundary and basic combinatorial topology implies that $R$ is simply connected in $M$. This implies that $R$ is a $3$-ball and so it can be used to reduce the number of intersections of $S$ and $\partial C_A$, contradicting normal position. Hence, $S'$ splits $C_B$.

Now there are two cases (see Figure \ref{fig3}). Suppose first that $v$ is an endpoint for at least two $b$-edge. This implies that $S'$ has at least two boundary components; each of which is contained in $\partial C_B$. Since these edges share the endpoint $v$, these boundary components co-bound the same component of $S' \cap C_A$. Let $d_1, d_2$ be two boundary components of $S'$ which are not separated by another such boundary component  of $S'$ in $\partial C_B$. Let $\alpha$ be an arc between $d_1$ and $d_2$ in $\partial C_B$ which intersects no other boundary component of $S'$ and let $\beta$ be an arc in $S'$ joining $d_1$ and $d_2$ with $\partial \beta = \partial \alpha$ that does not intersect $\partial  C_A$. Since $S$ and $\partial C_B$ are in normal position, $\beta$ is not homotopic relative endpoints into $\partial C_B$ and so $\gamma = \alpha * \beta$ is an essential loop in $C_A \cap C_B$ which is disjoint from $\partial C_A$ and can be homotoped not to intersect $S'$. Hence, if $[\gamma]$ denotes the conjugacy class of the smallest free factor containing $\langle \gamma \rangle$ then

\begin{eqnarray*}
d_{B}(\partial C_A, S) &\le& \text{diam}_{\mathcal{F}(B)}(S) + d_{B}(\partial C_A, S') \\
&\le& \text{diam}_{\mathcal{F}(B)}(S) + d_{B}(\partial C_A,  [\gamma])  + d_{B} ([ \gamma], S') \\
&\le& 4 + 4 + 4 = 12,
\end{eqnarray*}
a contradiction.

If $v$ is the endpoint of only one $b$-edge, this argument does not work. In this case, $S'$ is a disk and we argue as follows: first, any disk component of $S' \cap C_A$ splits $C_A$ and is disjoint from $\partial C_B$ providing the bound $d_{A}(\partial C_B, S) \le 4$, a contradiction. So assume that each components of $S' \cap C_A$ has at least two boundary components on $\partial C_A$, except possibly the unique component with a boundary component on $\partial C_B$. Among all components of $S' \cap C_A$ choose the component $S''$ which has two boundary components $d_1,d_2$ on $\partial C_A$ which are least separated by other components of $S' \cap \partial C_A$. Let $\alpha$ be an arc in $\partial C_A$ between $d_1$ and $d_2$ that intersects only the circles of $S' \cap \partial C_A$ that separate $d_1$ from $d_2$ in $\partial C_A$. Note that by our choice of $S''$ each circle of $S' \cap \partial C_A$ that is crossed by $\alpha$ bounds a distinct component of $S' \cap C_A$. Let $\beta$ be an arc in $S''$ joining these boundary components with $\partial \beta =\partial \alpha$. As before, $\gamma = \alpha * \beta$ is an essential loop in $C_A \cap C_B$ and we obtain a similar contradiction as above if $\gamma$ intersects $S'$ at most once; so suppose that this is not the case. Since intersections between $S'$ and $\gamma$ must occur along $\alpha$ we conclude that there is a component $C$ of $S' \cap C_A$ which does not have boundary on $\partial C_B$ and intersects $\gamma$ exactly once. This implies that $C$ is nonseparating in $C_A$. Hence, $C$ splits $C_A$ and is disjoint from $\partial C_B$. This provides the upper bound on distance 
\begin{eqnarray*}
d_A(S, \partial C_B) &\le& \mathrm{diam}_A(S) + d_A(C,\partial C_B) \\
&\le& 4+4 =8,
\end{eqnarray*}
a contradiction. 
\end{proof}

\begin{claim}\label{ABvertex}
There exists an $AB$-vertex of $G$ that has both an $a$-edge and a $b$-edge.
\end{claim}

\begin{proof}[Proof of claim \ref{ABvertex}]
Assume to the contrary; that is assume that no component of $S \cap C_A \cap C_B$ has its boundary on both $\partial C_A$ and $\partial C_B$. Let $s \in S$ be a sphere of $S$ that splits $C_B$, this sphere exists by assumption. If $s$ intersection $\partial C_B$ then it does not meet $\partial C_A$ and so $d_B(s,\partial C_A) \le 4$, a contraction. Hence, $s \subset C_B$. If $s$ also splits $C_A$, i.e. if some component of $s \cap C_A$ splits $C_A$, then we conclude $d_A(s, \partial C_B) \le 4$; so it must be the case that every component of $s \cap C_A$ is superfluous, that is, it separates $C_A$ and bounds to one side a component that is relatively simply connected. Note this implies in particular that no component of $s \cap C_A$ is a disk. We show that this also leads to a contraction. The argument is similar to that of the second part of Claim \ref{notrees}.

Among all components of $s \cap C_A$ choose the one with boundary components on $\partial C_A$ that are least separated by circles of $s \cap \partial C_A$, call this component $s''$. As in the poof of Claim \ref{ABvertex}, let $\alpha$ be an arc in $\partial C_A$ between these boundary components of $s''$ that intersects only the circles of $s \cap \partial C_A$ that separate these boundary components. Note that by our choice of $s''$ each circle of $s \cap \partial C_A$ that is crossed by $\alpha$ bounds a distinct component of $s \cap C_A$. Let $\beta$ be an arc in $s''$ joining these boundary components, with the same endpoints as $\alpha$. By normal position, $\gamma = \alpha * \beta$ is an essential loop in $C_A \cap C_B$ that can be homotoped to miss $s''$ and we obtain a similar contradiction as above if $\gamma$ does not intersect any other components of $s \cap C_A$; so assume that this is not the case. Since additional intersections with $s$ must occur along $\alpha$ we conclude that there is a component $C$ of $s \cap C_A$ that intersects $\gamma$ exactly once. This implies that $C$ is nonseparating in $C_A$ and contradicts the statement that all components of $s \cap C_A$ are superfluous. 
\end{proof}

To conclude the proof of the proposition, first locate an $AB$-vertex $v$ that has both an $a$-edge and a $b$-edge. The existence of $v$ is guaranteed by Claim \ref{ABvertex}. By the Claim \ref{notrees}, the $b$-edges at $v$ are not contained in a terminal $b$-tree. Hence, there is an $a$-edge adjacent to this $b$-tree in the complement of the initial vertex; the adjacency necessarily occurring at an $AB$-vertex. At this new vertex, Claim \ref{notrees} now implies that the $a$-edges are not contained in a terminal $a$-tree. Hence we may repeat the process and find a new $AB$-vertex to which we may again apply  Claim \ref{notrees}. Since $G$ is a forest, these $AB$ vertices are distinct and we conclude that $G$ is infinite. This contradicts that fact that edges of $G$ correspond to components of the intersection of transverse sphere systems $S$ and $S_A \cup S_B$ in $M_k$ and must, therefore, be finite.

\end{proof}

\section{Order on overlapping factors}
For trees $T,T' \in \mathcal{K}^0$ and $K \ge 2M +1$, define $\Omega(K,T,T')$ to be the set of (conjugacy classes of) free factors with the property that $A \in \Omega(K,T,T')$ if and only if $d_A(T,T')\ge K$. This definition is analogous to \cite{CLM}, where the authors put a partial order on the set of subsurfaces with large projection distance between two fixed markings. See \cite{MM2} and \cite{BKMM} for the origins of this ordering on subsurfaces. Defining a partial ordering on $\Omega(K,T,T')$, however, requires a more general notion of projection than is available in our situation. We resolve this issue by defining a relation which is thought of as an ordering, but is not necessarily transitive. That this is justified is the content of Lemma \ref{order}.

For $A, B \in \Omega(K,T,T')$ that \emph{overlap} we take $A \prec B$ to mean that $d_A(T,B) \ge M+1$, where $M$ is as in Proposition \ref{B}.  As noted above, this does not define a partial order. In particular, if $A \prec B$ and $B \prec C$ there is no reason to expect that $A$ and $C$ will meet as free factors. We do, however, have the following version of Proposition 3.6 from \cite{CLM}.

\begin{proposition} \label{factororder}
Let $K \ge 2M +1$ and choose $A,B \in \Omega(K,T,T')$ that overlap. Then $A$ and $B$ are ordered and the following are equivalent
\begin{multicols}{2}
\begin{enumerate}
\item $A \prec B$
\item $d_A(T,B) \ge M+1$
\item $d_B(T,A) \le M$
\item $d_B(T',A) \ge M+1$
\item $d_A(T',B) \le M$
\end{enumerate}
\end{multicols}
\end{proposition}

\begin{proof}
$(1)$ implies $(2)$ is by definition, $(2)$ implies $(3)$ is Proposition \ref{B}, $(3)$ implies $(4)$ is the observation that
$$d_B(T',A) \ge d_B(T,T') - d_B(T,A) \ge 2M+1 - M = M+1,$$
and the proofs of the remaining implications are similar. To show that $A,B \in \Omega(K,T,T')$ which overlap are ordered, note that by the equivalence of the above conditions if $A \nprec B$ then $d_A(T,B) \le M$ and if $B \nprec A$, switching the roles of $A$ and $B$, $d_A(T',B) \le M$ so that 
$$d_A(T,T') \le d_A (T,B) + d_A(B,T') \le 2M \le K,$$
a contradiction.
\end{proof}

\section{Normal forms in $A(\Gamma)$}
Let $\Gamma$ be a simplicial graph with vertex set $V(\Gamma) = \{s_1, \ldots, s_n \}$ and edge set $E(\Gamma) \subset V(\Gamma) \times V(\Gamma)$. The right-angled Artin group, $A(\Gamma)$, associated to $\Gamma$ is the group presented by
$$\langle s_i \in V(\Gamma): [s_i,s_j] =1 \iff (s_i,s_j) \in E(\Gamma) \rangle .$$
We refer to $s_1, \ldots, s_n$ as the \emph{standard generators} of $A(\Gamma)$.

\subsection{The \cite{CLM} partial order}
In this section, we briefly recall a normal form for elements of a right-angled Artin group. For details see Section $4$ of \cite{CLM} and the references provided there. Fix a word $w = x_1^{e_1} \ldots x_k^{e_k}$  in the vertex generators of $A(\Gamma)$, with $x_i \in \{s_1, \ldots, s_n \}$ for each $i=1, \ldots, k$. Each $x_i^{e_i}$ together with its index, which serves to distinguish between duplicate occurrences of the same generator, is a \emph{syllable} of the word $w$. Let $\syl(w)$ denote the set of syllables for the word $w$. We consider the following $3$ moves that can be applied to $w$ without altering the element in $A(\Gamma)$ it represents:
\begin{enumerate}
\item If $e_i = 0$, then remove the syllable $x_i^{e_i}$.
\item If $x_i = x_{i+1}$ as vertex generators, then replace $x_i^{e_i}x_{i+1}^{e_{i+1}}$ with $x_i^{e_i+e_{i+1}}$.
\item If the vertex generators $x_i$ and $x_{i+1}$ commute, then replace $x_i^{e_i}x_{i+1}^{e_{i+1}}$ with $x_{i+1}^{e_{i+1}}x_i^{e_i}$.
\end{enumerate}

For $g \in A(\Gamma)$, set Min$(g)$ equal to the set of words in the standard generators of $A(\Gamma)$ that have the fewest syllables among words representing $g$. We refers to words in Min$(g)$ as the \emph{normal form} representatives of $g$. Hermiller and Meiler showed in \cite{HerM} that any word representing $g$ can be brought to any word in Min$(g)$ by applications of the three moves above. Since these moves do not increase the word (or syllable) length, we see that words in Min$(g)$ are also minimal length with respect to the standard generators and that any two words in Min$(g)$ differ by repeated application of move $(3)$ only. It is verified in \cite{CLM} that for any $g \in A(\Gamma)$ and $w,w' \in \text{Min}(g)$ there is a natural bijection between $\syl(w)$ and $\syl(w')$ allowing one to define for $g \in A(\Gamma)$, $\syl(g) = \syl(w)$ for $w \in \text{Min}(g)$. For each $g \in A(\Gamma)$, this permits a strict partial order $\prec$ on the set $\syl(g)$ by setting $x_i^{e_i} \prec x_j^{e_j}$ if and only if for every $w \in \text{Min}(g)$ the syllable  $x_i^{e_i}$ precedes $x_j^{e_j}$ in the spelling of $w$. 

\subsection{Order on meeting syllables}
In analogy with the weaker notion of order on free factors, for $g \in A(\Gamma)$ let $\overset{m}{\prec}$ be the relation on $\syl(g)$ defined as follows: $x_i^{e_i} \overset{m}{\prec} x_j^{e_j}$ if and only if $x_i^{e_i} \prec x_j^{e_j}$ and there is a normal form $w \in \text{Min}(g)$ where $x_i^{e_i}$ and $x_j^{e_j}$ are adjacent. The following observation will be important in proving the lower bound on distance in our main theorem.

\begin{lemma} \label{trans}
The strict partial ordering $\prec$ on $\syl(g)$ is the transitive closure of the relation $\overset{m}{\prec}$.
\end{lemma}

\begin{proof}
From the definition of $\overset{m}{\prec}$ it suffices to show that if $x_i^{e_i} \prec x_j^{e_j}$ in $\syl(w)$ then $x_i^{e_i}$ and $x_j^{e_j}$ cobound a chain of syllables for which adjacent terms are ordered by $\overset{m}{\prec}$. To this end, let 
$$x_i^{e_i} = a_1 \prec a_2 \prec \ldots \prec a_n = x_j^{e_j} $$
be a chain of maximal length joining $x_i^{e_i}$ and $x_j^{e_j}$ in $\syl(g)$. We show that each pair of consecutive terms in the chain is ordered by $\overset{m}{\prec}$. Take $1\le i\le n$ and consider the $w\in \mathrm{Min}(g)$ for which $a_i$ and $a_{i+1}$ are separated by the least number of syllables in $w$. If $a_i$ and $a_{i+1}$ are adjacent in $w$ we are done, otherwise write
$$w=w_1 \cdot a_i \cdot s \cdot w_2 \cdot a_{i+1} \cdot w_3$$
where $w_1,w_2,w_3$ are possibly empty subwords of $w$ and $s$ is a syllable of $w$. By our choice of $w$, $a_i \prec s$, for otherwise we could commute $s$ past $a_i$ resulting in a normal form for $g$ with fewer syllables separating $a_i$ and $a_{i+1}$.  Then either $s \prec a_{i+1}$, which contradicts the assumption that the chain is maximal, or $s$ can be commuted past $a_{i+1}$ resulting in a normal form $w' \in \mathrm{Min}(g)$ with
$$w' = w_1 \cdot a_i \cdot w_2' \cdot a_{i+1} \cdot s \cdot w_3' $$
where $w_2'$ is a subword of $w_2$. This contradicts our choice of $w$. Hence, $a_i$ and $a_j$ must occur consecutively in $w$ and so $a_i \overset{m}{\prec} a_{i+1}$ as required.
\end{proof}

\section{Large projection distance}
Fix an admissible system $ \mathcal{S} = (\mathcal{A},\{f_i\})$ for $\mathbb{F}$ with coincidence graph $\Gamma$. This determines a homomorphism $\phi=\phi_{\mathcal{S}}: A(\Gamma) \to \Out(\mathbb{F})$ by mapping the vertex generator $s_i$ to the outer automorphism $f_i$.   

For $g \in A(\Gamma)$ with $w = x_1^{e_1} \dots x_k^{e_k} \in \text{Min}(g)$, let  $J: \{ 1, \ldots,k\} \to \{ 1, \ldots, n \}$ be defined so that $x_i = s_{J(i)}$, as generators of $A(\Gamma)$. Hence, $\phi(x_i) = f_{J(i)}$ is supported on $A_{J(i)}$. Write
$$A^w(x_i^{e_i})  = \phi(x_1^{e_1} \ldots x_{i-1}^{e_{i-1}})(A_{J(i)})$$
for $i = 2, \ldots, k$ and $A^w(x_1^{e_1}) = A_{J(1)}$. This defines a map
$$A^w : \text{syl}(w) \to \mathcal{FF} .$$
It is verified in \cite{CLM} that this map is well-defined for $g \in A(\Gamma)$, independent of the choice of normal form. 
Then, set $A^{g} := A^w$ for $w \in \text{Min}(g)$ and set $\fact(g)$ equal to the image of the map $A^g:  \text{syl}(g) \to \mathcal{FF}$. We refer to the free factors in $\fact(g)$ as the \emph{active free factors} for $g \in A(\Gamma)$.  For notional convenience, set $B_i = A_{J(i)}$ and $g_i = \phi(x_i^{e_i}) = f^{e_i}_{J(i)}$. Note that this notation is for a fixed $w \in \text{Min}(g)$.

Having developed the necessary tools in the free group setting, the proof of the first part of the following theorem is a verification that the arguments of \cite{CLM} extend to this situation, even with a weaker form of Proposition \ref{B} . We repeat their argument here for completeness. Let $M$ be the constant determined in Proposition \ref{B} and let $L = 4$ be the Lipschitz constant for the projection $\pi_A: \K \to\mathcal{F}(A)$, $A\in \FF$. In the appendix we determine the corresponding Lipschitz constant for $\mathcal{X} \to \mathcal{F}(A)$. 


\begin{theorem} \label{factordistance}
Given admissible collection $\mathcal{A}$ of free factors for $\mathbb{F}$ with coincidence graph $\Gamma$ and $T \in \mathcal{K}^0$, there is a $K \ge 5M+3L$ so that if outer automorphisms $\{f_i\}$ are chosen to make $(\mathcal{A},\{f_i\})$ an admissible system with $\ell_{A_i}(f_i) \ge 2K$ then the induced homomorphism $\phi: A(\Gamma) \to \Out(\mathbb{F})$ satisfies the following: For any $g \in A(\Gamma)$ with normal form $w = x_1^{e_1} \ldots x_k^{e_k} \in \mathrm{Min}(g)$,
\begin{enumerate}
\item $d_{A^g(x_i^{e_i})} (T, \phi(g)T) \ge K|e_i|$ for $1 \le i \le k$. 
In particular, $\fact(g) \subset \Omega(K,T,\phi (g) T)$.

\item If $x_i^{e_i} \overset{m}{\prec} x_j^{e_j}$, then $A^g(x_i^{e_i})$ and $A^g(x_j^{e_j})$ overlap and 
$$A^g(x_i^{e_i} )\prec A^g(x_j^{e_j}).$$
\end{enumerate}

\end{theorem}

\begin{proof}

Set $K = 5M +3L + 2\cdot \text{max} \{ d_{A_i}(T,A_j) \}$ and observe that this choice of $K$ has the property that if $A_i$ and $A_j$ overlap then $d_{A_i}(T,A_j) \le K/2 -M$. The proof of $(1)$ is by induction on the syllable length of $w \in \mathrm{Min}(g)$.  If $w$ has only one syllable then 
$$d_{A_{J(1)}}(T,f^{e_1}_{J(1)}T) \ge \ell_{A_{J(1)}}(f^{e_1}_{J(1)}) \ge 2K|e_1|.$$

Now suppose that $(1)$ has been proven for all elements in $A(\Gamma)$ that have representative with less than or equal to $k-1$ syllables. Take $g \in A(\Gamma)$ with $w = x_1^{e_1} \dots x_k^{e_k} $ a $k$-syllable normal form representative for $g$. Using the notation at the beginning of this section, write $\phi(w)$ as $g_1 \ldots g_k$ so that for  $1 \le i \le k$ we must show 
$$d_{g_1 \ldots g_{i-1}B_i} (T, g_1 \ldots g_k T) \ge K|e_i|.$$

With $x_i^{e_i} \in \mathrm{syl}(g)$ fixed and $g_i = \phi(x_i^{e_i})$, we write $\phi(g)$ as $abg_ic$ by choosing a normal form $w \in \text{Min}(g)$ so that
\begin{enumerate}
\item $c = g_{i+1} \ldots g_k$ and $g_i$ and $g_{i+1}$ do not commute,
\item $a = g_1 \ldots g_l$ with $l$ the largest index among $w \in \text{Min}(g)$ so that $g_l$ and $g_i$ do not commute, and
\item $b = g_{l+1}\ldots g_{i-1}$, all of which commute with $g_i$. 
\end{enumerate}
Note that we allow $a,b$ or $c$ to be empty.

Using this notation, we show that $d_{abB_i}(T,abg_icT) \ge K|e_i|$. By Lemma \ref{properties2} and the triangle inequality,
\begin{eqnarray}
d_{abB_i}(T,abg_icT) &=& d_{B_i}(b^{-1}a^{-1}T, g_icT) \\
&\ge& d_{B_i}(T, g_iT ) - d_{B_i}(b^{-1}a^{-1}T,T ) -d_{B_i}(g_icT, g_iT ). 
\end{eqnarray}

Since $b$ is written in terms of generators that restrict to the identity outer automorphism on $B_i$ and $g_i$ restricts to an isometry of the free factor complex of $B_i$, Lemma \ref{properties2} implies
$$d_{B_i}(b^{-1}a^{-1}T,T ) = d_{B_i}(a^{-1}T,T)$$
and
$$d_{B_i}(g_icT, g_iT ) =d_{B_i}(cT, T ). $$

This, along with our hypothesis on translation length, allows us to write
\begin{eqnarray}
d_{abB_i}(T,abg_icT) \ge 2K|e_i| - d_{B_i}(a^{-1}T,T) - d_{B_i}(cT, T ).
\end{eqnarray}
We use the induction hypotheses to show that both terms subtracted in $(3)$ are $\le K/2$. This will complete the proof of (1). First, observe that each of  $a^{-1} = g_l^{-1} \ldots g_1^{-1}$ and $c = g_{i+1} \ldots g_k$ is either trivial or is the image of a normal form subword of $w$ with strictly fewer than $k$ syllables and begins with a syllable not commuting with $x_i^{e_i}$. This is all that is needed for the remainder of the proof. We show the inequality $d_{B_i}(a^{-1}T,T) \le K/2$, the other appears in \cite{CLM} where the proof follows through without change.

By the induction hypothesis applied to $a^{-1}$,
$$d_{B_{l}}(T,a^{-1}T) = d_{B_{l}}(T,g_l^{-1} \ldots g_1^{-1}T)  \ge K|e_{l}|,$$
and so since $d_{B_{l}}(T,B_i) \le K/2 -M$ by our choice of $K$, we have $d_{B_{l}}(B_i,a^{-1}T)\ge K-(K/2 -M) \ge M+1$. Since $B_i$ and $B_l$ overlap, Proposition \ref{B} implies that $d_{B_i}(B_l, a^{-1}T) \le M$, so by another application of $d_{B_{i}}(T,B_l) \le K/2 -M$,
$$d_{B_i}(a^{-1}T,T) \le M+(K/2 -M) \le K/2 ,$$
as required. This completes the proof of the first part of the theorem.

The second part of the theorem is also proven by induction on syllable length. If $g \in A(\Gamma)$ has syllable length equal to $1$, then there is nothing to prove. Suppose that the ordering statement holds for all $g$ with a minimal syllable representative with less then or equal to $k-1$ syllables. As in the first part of the proof, take $g \in A(\Gamma)$ with $w = x_1^{e_1} \dots x_k^{e_k} $ a $k$-syllable normal form representative for $g$. Write $\phi(w)$ as $g_1 \ldots g_k$ and suppose that $x_i^{e_i} \overset{m}{\prec} x_j^{e_j}$ as syllables of $g$. If $j \le k-1$ then we may apply the induction hypothesis to a prefix of $w$ and conclude $A^w(x_i^{e_i} )\prec A^w(x_j^{e_j})$. More precisely, let $w'$ be the word formed by the first $k-1$ syllables of $w$; this is a normal form word for some $g' \in A(\Gamma)$. By the induction hypothesis $A^{w'}(x_i^{e_i} )$ and $A^{w'}(x_j^{e_j})$ overlap and $A^{w'}(x_i^{e_i} )\prec A^{w'}(x_j^{e_j})$. This suffices since for $l\le k-1$ we have $A^{w}(x_l^{e_l})= A^{w'}(x_l^{e_l})$, using the obvious identification of the syllables of $w'$ with those of $w$.

Otherwise, $j=k$ and by definition of $\overset{m}{\prec}$ we may choose $w\in \text{Min}(g)$ so that $w = ax_i^{e_i}x_k^{e_k}$ and so $\phi(w) = \phi(a)g_ig_k$. Since $x_i^{e_i} \overset{m}{\prec} x_k^{e_k}$, $B_i$ and $B_k$ overlap and so $\phi(a)g_iB_i = \phi(a)B_i$  and $\phi(a)g_iB_k$ also overlap. We have
\begin{eqnarray*}
d_{A^g(x_k^{e_k})} (A^g(x_i^{e_i}), \phi(g)T) 
&=& d_{\phi(a)g_iB_k}(\phi(a)B_i, \phi(a)g_ig_kT) \\
&=& d_{B_k}(B_i, g_kT) \\
&\ge&  d_{B_k}(T, g_kT) - d_{B_k}(B_i, T) \\
&\ge& d_{A_{J(k)}}(T,f^{e_k}_{J(k)}T)  - d_{A_{J(k)}}(A_{J(i)}, T) \\
&\ge& 2K - K  \\
&\ge& M+1,
\end{eqnarray*}
and so since $A^g(x_i^{e_i}), A^g(x_k^{e_k}) \in \Omega(K,T,\phi(g)T)$, by Proposition \ref{factororder}
$$A^w(x_i^{e_i} )\prec A^w(x_k^{e_k}) .$$
\end{proof}

\section{The lower bound on distance for admissible systems}\label{lower}
Take $K \ge 5M +3L$ as in Theorem \ref{factordistance}.

Let $\mathcal{A} = (\{A_i\}, \{f_i\})$ be an admissible system satisfying the hypotheses of Theorem \ref{factordistance} for $T \in \mathcal{K}^0$. For $g \in A(\Gamma)$ and $w \in \text{Min}(g)$ write in normal form 
$$w = x_1^{e_1} \ldots x_k^{e_k}.$$
We make use of the notation introduced at the beginning of the previous section.

Set $T' = \phi(g)T$ and choose a geodesic $T =T_0, T_1, \ldots, T_N = T'$ in the $1$-skeleton of  $\mathcal{K}_n$. Similar to \cite{MM2}, we define the subinterval $I_A =[a_A,b_A] \subset [0,N]$ associated to the free factor $A \in \Omega(K,T,T')$ as follows: Set
$$a_A = \max\{k\in \{0, \ldots, N\}: d_A(T,T_k)\le 2M+L \} $$
and 
$$b_A =  \min\{k\in \{a_A, \ldots, N\}: d_A(T_k,T')\le 2M+L \}.$$
Since $A \in \Omega(K,T,T')$, $d_A(T,T')\ge K \ge 5M+3L$ and so both $a_A$ and $b_A$ are well-defined and not equal. Hence, the interval $I_A$ is nonempty and for all $k \in I_A$, 
$$d_A (T_k,T) \ge 2M+1 \quad \mathrm{and} \quad d_A(T_k,T') \ge 2M+1.$$
This uses that fact that the projection from $\K^0$ to $\mathcal{F}(A)$ is $L$-Lipschitz. The next lemma shows that if syllables are ordered, then distance in their associated free factors cannot be made simultaneously. 

\begin{lemma} \label{order}
With notation fixed as above, if $x_i^{e_i}, x_j^{e_j} \in \syl(w)$ and $x_i^{e_i} \prec x_j^{e_j} $ then $$I_{A^w(x_i^{e_i})} < I_{A^w(x_j^{e_j})}.$$
That is, the intervals are disjoint and correctly ordered in $[0,N]$. 
\end{lemma}

\begin{proof}
We first prove the proposition when $x_i^{e_i} \overset{m}{\prec} x_j^{e_j} $. Recall that since $x_i^{e_i} \overset{m}{\prec} x_j^{e_j} $, Theorem \ref{factordistance} implies that the free factors $A^w(x_i^{e_i})$ and $A^w(x_j^{e_j})$ overlap and are ordered, $A^w(x_i^{e_i}) \prec A^w(x_j^{e_j})$. If  $k \in I_{A^w(x_i^{e_i})}$, then $d_{A^w(x_i^{e_i})}(T_k,T') \ge 2M+1$ and since $A^w(x_i^{e_i}) \prec A^w(x_j^{e_j})$ we have $d_{A^w(x_i^{e_i})}(A^w(x_j^{e_j}),T') \le M$. The triangle inequality then implies that
$$d_{A^w(x_i^{e_i})}(T_k, A^w(x_j^{e_j})) \ge M+1.$$

As the free factors $A^w(x_i^{e_i})$ and $A^w(x_j^{e_j})$ overlap, by Proposition \ref{B} we have $$d_{A^w(x_j^{e_j})}(T_k, A^w(x_i^{e_i})) \le M.$$
Combining this with the inequality $d_{A^w(x_j^{e_j})}(A^w(x_i^{e_i}),T) \le M$, again coming from the ordering, provides
$$d_{A^w(x_j^{e_j})}(T, T_k) \le 2M.$$
Since this is true for each $k \in I_{A^w(x_i^{e_i})}$ it follows from the definition of $I_{A^w(x_j^{e_j})}$ that $I_{A^w(x_i^{e_i})} \cap I_{A^w(x_j^{e_j})} = \emptyset$. So if there were an index $k \in I_{A^w(x_i^{e_i})}$ with $k > a_{A^w(x_j^{e_j})}$ then by disjointness of the intervals $a_{A^w(x_i^{e_i})}> a_{A^w(x_j^{e_j})}$. This contradiction the choice of $a_{A^w(x_j^{e_j})}$ as the largest index $k$ with $d_{A^w(x_j^{e_j})}(T,T_k) \le 2M+1$ and shows that the intervals of interest are disjoint and ordered as $I_{A^w(x_i^{e_i})} < I_{A^w(x_j^{e_j})}$.

Now, if more generally we have that  $x_i^{e_i} \prec x_j^{e_j} $, then by Lemma \ref{trans}, $x_i^{e_i}$ and $x_j^{e_j} $ can be joined by a chain of syllables 
$$x_i^{e_i} = a_0 \overset{m}{\prec} a_1 \overset{m}{\prec} \ldots \overset{m}{\prec} a_l = x_j^{e_j} .$$ 
Hence, we conclude 
$$I_{A^w(x_i^{e_i})} < I_{A^w(a_1)}< \ldots < I_{A^w(x_j^{e_j})}, $$
as required.
\end{proof}

Let $s = s(\Gamma)$ be the size of the largest complete subgraph of $\Gamma$. This is also the maximal rank of a free abelian subgroup of $A(\Gamma)$. Note by the definition of an admissible system, $s$ is bounded above by a constant depending only on the rank of $\mathbb{F}$. To simplify notations, associated to the free factor $A^g(x_i^{e_i})$ we set $a_i = a_{A^g(x_i^{e_i})}$ and $b_i = b_{A^g(x_i^{e_i})}$.

\begin{lemma}[Lower bound on distance] \label{lowerbound}
With notation fixes as above, $K$ as in Theorem \ref{factordistance} and $w \in \mathrm{Min}(g)$ in normal form 
$$\sum_{1 \le i\le k} d_{A^g(x_i^{e_i})}(T,\phi(g)T) \le 5sL \cdot d_{\mathcal{K}}(T,\phi(g)T).$$
\end{lemma}

\begin{proof}
Since $A^g(x_i^{e_i}) \in \Omega(K,T,\phi(g)T)$ for all $x_i^{e_i} \in \syl(g)$ by Theorem \ref{factordistance}, we have the collection of nonempty subintervals $\{I_{A^g(x_i^{e_i})}: 1 \le i \le k \}$ of $\{0,1, \ldots, N\}$. If, for $i \le j$, it is the case that $x_i^{e_i} \prec x_j^{e_j}$ then by Lemma \ref{order}, $I_{A^g(x_i^{e_i})}$ and $I_{A^g(x_j^{e_j})}$ are ordered and, hence, disjoint.  Further, any collection of syllables pairwise unordered by $\prec$ has size bounded above by $s$. This is clear since such a collection of syllables can be commuted to be consecutive in $w$ using move $(3)$ and so correspond to distinct pairwise commuting standard generators. We conclude that for no more then $s$ indices $i = 1, \ldots, k$ is any integer in $\{0,1, \ldots, N\}$ contained in $I_{A^g(x_i^{e_i})}$. Hence,
$$\sum_{1 \le i \le k} | b_i - a_i | \le s \cdot d_{\mathcal{K}}(T, \phi(w)T) .$$

Using the Lipschitz condition on the projections and the triangle inequality,
\begin{eqnarray*}
d_{A^g(x_i^{e_i})}(T, \phi(g)T) &\le& d_{A^g(x_i^{e_i})}(T_{a_i}, T_{b_i})+ 4M +2L \\
&\le& L|b_i - a_i| +4M +2L.
\end{eqnarray*}
Since for each $A \in \Omega(K,T,\phi(g)T)$, $d_A(T,\phi(g)T)\ge K \ge 5M+3L$ we have $|b_A - a_A| \ge \frac{M+L}{L}$. This implies that $d_A(T,\phi(g)T) \le 5L \cdot |b_A - a_A|$ and so putting this with the inequality above
$$\sum_{1 \le i\le k} d_{A^g(i)}(T,\phi(g)T) \le 5sL \cdot d_{\mathcal{K}}(T,\phi(g)T),$$
as required.
\end{proof}

\section{The quasi-isometric embedding}
We can now prove Theorem \ref{main}.

\begin{theorem}\label{main}
Given an admissible collection $\mathcal{A}$ of free factors for $\mathbb{F}_n$ with coincidence graph $\Gamma$ there is a $C \ge 0$ so that if outer automorphism $\{f_i\}$ are chosen making $\mathcal{S} = (\mathcal{A},\{f_i\})$ an admissible system with $\ell_{A_i}(f_i) \ge C$ then the induced homomorphism $\phi= \phi_{\mathcal{S}}: A(\Gamma) \to \Out(\mathbb{F}_n)$ is a quasi-isometric embedding.
\end{theorem}

\begin{proof}
Suppose $\mathcal{A}$ is an admissible collection of free factors and $T \in \mathcal{K}^0$. Take $C =2K$, for $K$ as in Theorem \ref{factordistance}. We show that the orbit map $A(\Gamma) \to \mathcal{K}_n^1$ 
$$g \mapsto \phi(g)T$$
is a quasi-isometric embedding, where $A(\Gamma)$ is given the word metric in its standard generators. Since $\Out(\mathbb{F}_n)$ is quasi-isometric to $\mathcal{K}_n^1$, this suffices to prove the theorem. First, recall that the orbit map is Lipschitz, as is any orbit map induced by an isometric action of a finitely generated group on a metric space. Specifically, $d_{\K}(T,\phi(g)T) \le A\cdot |g|$, where $A = \mathrm{max}\{ d_{\K}(T, \phi(s_i)T : 1 \le i\le n\}$ and $s_1, \ldots, s_n$ are the standard generators.

Let $g\in A(\Gamma)$. By Theorem \ref{factordistance}, we know that if $w = x_1^{e_1} \ldots x_k^{e_k} \in \text{Min}(g)$, then
$$d_{A^g(x_i^{e_i})} (T, \phi(g)T) \ge K|e_i|$$
for $1 \le i \le k$. Hence, by Lemma \ref{lowerbound}
\begin{eqnarray*}
|g| &=& \sum_{1\le i \le k} |e_i| \\
&\le& \frac{1}{K} \sum_{1 \le i\le k} d_{A^g(x_i^{e_i})}(T,\phi(g)T) \\
&\le& \frac{5sL}{K} \cdot d_{\mathcal{K}}(T,\phi(g)T).
\end{eqnarray*}

We conclude that for any $g,h \in A(\Gamma)$
\begin{eqnarray*}
\frac{1}{A} d_{\K}(\phi(g)T,\phi(h)T) &=& \frac{1}{A} d_{\K}(T, \phi(g^{-1}h)T) \le |g^{-1}h|  = d_{A(\Gamma)}(g,h) 
\end{eqnarray*}
and
\begin{eqnarray*}
 d_{A(\Gamma)}(g,h)) \le \frac{5sL}{K} \cdot d_{\mathcal{K}}(T,\phi(g^{-1}h))T) = \frac{5sL}{K} \cdot d_{\mathcal{K}}(\phi(g)T,\phi(h))T),
\end{eqnarray*}
as required.
\end{proof}

\section{Appendix: quasi-isometric embeddings into Outer space}\label{appen}
Here we adapt the arguments of the previous sections to show that if $K$ in Proposition \ref{factordistance} is chosen appropriately then the homomorphism induced from an admissible system gives a quasi-isometric orbit map from $A(\Gamma)$ into $\mathcal{X}_n$, (projectivized) Outer space with its Lipschitz metric. Since this section is independent from the rest of the paper, we assume that the reader has some familiarity with the metric properties of $\mathcal{X}_n$. See \cite{FMout, BFhyp, BFproj} for background. Let $d_{\mathcal{X}_n}$ denote the Lipschitz metric on $\mathcal{X}_n$ and for $G \in  \mathcal{X}_n$ and a free factor $A$ of $\F_n$, we use the notation $A|G \in \mathcal{X}(A)$ to denote the core of the cover of $G$ corresponding to $A$. This is the quotient graph of the proper splitting $\pi_{\S(A)}(\tilde{G})$ with edge lengths induced from $G$.

\begin{lemma}\label{Liptoout}
There is an $L$ such that for any $G,G' \in \X_n$ and $A \in \FF_n$
$$d_A(G,G') \le L \cdot d_{\mathcal{X}_n}(G,G') +L. $$
\end{lemma}

\begin{proof}
This follows essentially from Lemma $3.1$ of \cite{BFproj}. We provide some details using the notation found there. Let $\phi: G \to G'$ be an optimal map with induced path $G_t$, $t\in [\alpha,\omega]$. Since paths within a simplex of $\mathcal{X}_n$ do not affect the splitting, and so do not affect the projection to $\mathcal{F}(A)$, we assume that $G_t$ is a folding path. By Lemma $3.1$ of \cite{BFproj}, there is a subinterval $[\beta, \gamma) \subset [\alpha,\omega]$ so that for $t \in [\beta,\gamma)$,  $A|G_t$ has bounded volume and (after rescaling and reparameterizing) is a folding path in $\mathcal{X}(A)$. Moreover, the image in $\mathcal{X}_n$ of each component of the complement of $[\beta,\gamma)$ in $[\alpha,\omega]$ has projection diameter in $\S(A)$ bounded by $D$, a constant depending only on $n$. 

By definition of a folding path and properties of the subinterval $[\beta, \gamma)$, there is an optimal map $\phi_{\beta\gamma}:G_{\beta} \to G_{\gamma}$ which lifts to a map $\phi_A: A|G_{\beta} \to A|G_{\gamma}$. Hence,
\begin{eqnarray*}
\log d_{\mathcal{X}(A)}(A|G_{\beta},A|G_{\gamma}) &\le& \mathrm{Lip} \left( \phi_A: \frac{ A|G_{\beta}}{\mathrm{Vol} ( A|G_{\beta})}\to \frac{ A|G_{\gamma}}{\mathrm{Vol} ( A|G_{\gamma})} \right) \\
&\le& \frac{\mathrm{Vol} ( A|G_{\gamma})}{\mathrm{Vol} ( A|G_{\beta})}\cdot \mathrm{Lip}(\phi_A:A|G_{\beta} \to A|G_{\gamma}) \\
&\le& \frac{\mathrm{Vol} ( A|G_{\gamma})}{\mathrm{Vol} ( A|G_{\beta})}\cdot \mathrm{Lip}(\phi_{\beta\gamma}:G_{\beta} \to G_{\gamma} ) \\
&=& \frac{\mathrm{Vol} ( A|G_{\gamma})}{\mathrm{Vol} ( A|G_{\beta})}\cdot \log d_{\mathcal{X}_n}(G_{\beta}, G_{\gamma} ),
\end{eqnarray*}
where the ratio $\frac{\mathrm{Vol} ( A|G_{\gamma})}{\mathrm{Vol} ( A|G_{\beta})}$ is bounded by some constant $D'$ depending only on $n$. By Corollary 3.5 of \cite{BFhyp}, the projection $\mathcal{X}(A) \to \mathcal{F}(A)$ is $L'$-Lipschitz. We conclude
\begin{eqnarray*}
d_A(G,G') &\le& d_A(G,G_{\beta}) + d_A(G_{\gamma}, G') + d_A(G_{\beta},G_{\gamma})\\
&\le& 8+2D + L' \cdot d_{\mathcal{X}(A)}(G_{\beta},G_{\gamma}) \\
&\le& 8+2D+ L'e^{D'}  \cdot d_{\mathcal{X}_n}(G_{\beta}, G_{\gamma}) \\
&\le& 8+2D+ L'e^{D'}  \cdot d_{\mathcal{X}_n}(G, G').
\end{eqnarray*}
Setting $L = \max\{8+2D, L'e^{D'} \}$ completes the proof.

\end{proof}

Let $\mathcal{A}$ be an admissible collection of free factors with coincidence graph $\Gamma$ and for $T \in \K_n$, let $K$ be as in Theorem \ref{factordistance} where $L$ is taken from Lemma \ref{Liptoout}. By choosing outer automorphisms $\{f_i\}$ so that $\mathcal{S} = (\mathcal{A},\{f_i\})$ an admissible system with $\ell_{A_i}(f_i) \ge 2K$  the induced homomorphism $\phi: A(\Gamma) \to \Out(\mathbb{F}_n)$ is a quasi-isometric embedding by Theorem \ref{main}. This implies that $\phi$ is injective and permits us to identify $A(\Gamma)$ with its image in $\Out(\F_n)$. Fix $G \in \mathcal{X}_n$ so that the universal cover $\tilde{G}$ is conjugate as a free splitting to $T$. 

Recalling the notation of Section \ref{lower} for $g \in A(\Gamma)$ and $w \in \mathrm{Min}(g)$, set $G' = gG$ and choose a unit speed geodesic $G_t$,  $t\in [0,d_{\mathcal{X}_n}(G,G')]$, joining $G$ and $G'$ that is a folding path in $\X_n$. Define the subinterval $I_A =[a_A,b_A] \subset [0,d_{\mathcal{X}_n}(G,G')]$ associated to the free factor $A \in \Omega(K,G,G')$ as follows:
$$a_A = \sup\{t \in [0,d_{\mathcal{X}_n}(G,G']: d_A(G,G_t)\le 2M+L \} $$
and 
$$b_A =  \inf\{t \in [a_A,d_{\mathcal{X}_n}(G,G']: d_A(G_t,G')\le 2M+L \}.$$
Observe that for all $t \in I_A$, $d_A (T_t,T), d_A(T_t,T') \ge 2M+1$ since the projection from $\mathcal{X}_n$ to $\mathcal{F}(A)$ is $L$-Lipschitz. 

\begin{remark}
Bestvina and Feighn have show that the projection of a folding path in $\mathcal{X}_n$ to the free splitting complex (or free factor complex) of a free factor is an unparameterized quasi-geodesic with uniform constants \cite{BFproj}. This implies that the projection of a folding path has bounded backtracking and so outside of the interval $I_A$ the projection of the folding path $G_t$ to $\mathcal{F}(A)$ makes uniformly bounded progress.
\end{remark}

With this setup, the proof of Lemma \ref{order} follow exactly as stated when using the intervals $I_A$ for $A \in \Omega(K,G,gG)$ defined in this section. Alternatively, the ordering follows from the remark above. Once this is established, Lemma \ref{lowerbound} follows for the intervals of this section with $L$ as in Lemma \ref{Liptoout}. We can now prove the following:

\begin{theorem}
Given an admissible collection $\mathcal{A}$ of free factors for $\mathbb{F}_n$ with coincidence graph $\Gamma$ there is a $C \ge 0$ so that if outer automorphism $\{f_i\}$ are chosen making $\mathcal{S} = (\mathcal{A},\{f_i\})$ an admissible system with $\ell_{A_i}(f_i) \ge C$ then the induced homomorphism $\phi= \phi_{\mathcal{S}}: A(\Gamma) \to \Out(\mathbb{F}_n)$ is a quasi-isometric embedding and the orbit map into Outer space is a quasi-isometric embedding.
\end{theorem}

\begin{proof}
The only statement yet to be proven is that the orbit map into $\mathcal{X}_n$ is a quasi-isometric embedding. Fix $G \in \mathcal{X}_n$ and $K$ as determined in this section; set $C = 2K$. We have the orbit map $A(\Gamma) \to \mathcal{X}_n$ given by 
$$g \mapsto g G.$$
For the standard generators $s_1, \ldots, s_n$ of $A(\Gamma)$ set $A = \mathrm{max}\{d_{\mathcal{X}_n}(G,s_iG): 1\le i\le n) \}$. As in the proof of Theorem \ref{main} for any $g \in A(\Gamma)$, $d_{\mathcal{X}_n}(G,gG) \le A \cdot |g|$ where $|g|$ is the length of $g$ in the standard generators. Hence, the orbit map is $A$-Lipschitz.

Let $g\in A(\Gamma)$. By Theorem \ref{factordistance}, for $w = x_1^{e_1} \ldots x_k^{e_k} \in \text{Min}(g)$,
$$d_{A^g(x_i^{e_i})} (G, gG) \ge K|e_i|$$
for $1 \le i \le k$. Hence, by the $\mathcal{X}_n$-version of Lemma \ref{lowerbound}
\begin{eqnarray*}
|g| &=& \sum_{1\le i \le k} |e_i| \\
&\le& \frac{1}{K} \sum_{1 \le i\le k} d_{A^g(x_i^{e_i})}(G,gG) \\
&\le& \frac{5sL}{K} \cdot d_{\mathcal{X}_n}(G,gG),
\end{eqnarray*}
where $L$ is as in Lemma \ref{Liptoout}. This completes the proof.

\end{proof}

\bibliographystyle{amsalpha}

\bibliography{RAAGsandOut.bbl}

\end{document}